\documentclass{amsart}
\usepackage{graphicx}
\usepackage{verbatim}
\usepackage{textcomp}
\usepackage{amssymb}
\usepackage{cite}
\newtheorem{thm}{Theorem}[section]
\newtheorem{cor}[thm]{Corollary}
\newtheorem{lem}[thm]{Lemma}
\newtheorem{prop}[thm]{Proposition}
\newtheorem{conj}[thm]{Conjecture}

\theoremstyle{definition}

\theoremstyle{remark}

\numberwithin{equation}{section}
\newcommand{\norm}[1]{\left\Vert#1\right\Vert}
\newcommand{\abs}[1]{\left\vert#1\right\vert}
\newcommand{\Real}{\mathbb R}
\newcommand{\Ker}{\mathcal G}
\newcommand{\eps}{\varepsilon}
\newcommand{\lam}{\lambda}
\newcommand{\To}{\longrightarrow}
\newcommand{\ra}{\rangle}
\newcommand{\la}{\langle}

\title[Uniqueness of Self-similar Shrinkers]{Uniqueness of Self-similar Shrinkers with Asymptotically Conical Ends}
\author{Lu Wang}%
\address{Department of Mathematics, Johns Hopkins University\\
3400 N. Charles Street, Baltimore, MD 21218\\
and Mathematical Sciences Research Institute\\
17 Gauss Way, Berkeley, CA 94720.}
\email{coral0426@gmail.com}

\begin{document}
\begin{abstract}
Let $C\subset\Real^{n+1}$ be a regular cone with vertex at the origin. In this paper, we show the uniqueness for smooth properly embedded self-shrinking ends in $\Real^{n+1}$ that are asymptotic to $C$. As an application, we prove that not every regular cone with vertex at the origin has a smooth complete properly embedded self-shrinker asymptotic to it.
\end{abstract}



\thanks{The author was supported by the postdoctoral fellowship at MSRI}   

\maketitle

\section{Introduction}

Self-shrinkers are a special class of solutions to the mean curvature flow in $\Real^{n+1}$, in which a later time slice is a scaled down copy of an earlier slice. More precisely, a hypersurface $\Sigma$ in $\Real^{n+1}$ is said to be a self-shrinker if it satisfies 
\begin{equation}
\label{SelfshrinkerEqn}
H=\frac{1}{2}\langle x,\textbf{n}\rangle.
\end{equation} 
Here $H=\text{div}\left(\textbf{n}\right)$ is the mean curvature,  $\textbf{n}$ is the outward unit normal, $x$ is the position vector and $\langle,\rangle$ denotes the Euclidean inner product. One reason that self-shrinking solutions to the mean curvature flow are particularly interesting is that they provide singularity models of the flow; see \cite{H1,H2}, \cite{Il1} and \cite{Wh}.

Throughout, $O$ is the origin of $\Real^{n+1}$; $B_R$ denotes the open ball in $\Real^{n+1}$ centered at $O$ with radius $R$ and $S_R=\partial B_R$; $D_R$ denotes the open disk in $\Real^n\times\{0\}$ centered at the origin with radius $R$. We say that $C\subset\Real^{n+1}$ is a regular cone with vertex at $O$, if
\begin{equation}
\label{RegConeEqn}
C=\left\{l\Gamma,\ 0\le l<\infty\right\},
\end{equation}
where $\Gamma$ is a smooth connected closed (compact without boundary) embedded submanifold of $S_1$ with codimension one. Note that the normal component of the position vector on $C$ vanishes and $C\setminus\{O\}$ is smooth. In this paper, we study the uniqueness for smooth properly embedded self-shrinking ends in $\Real^{n+1}$ that are asymptotic to a given regular cone. Namely, we show:

\begin{thm}
\label{UniqueThm}
Let $C\subset\Real^{n+1}$ be a regular cone with vertex at $O$ and $R_0$ a positive constant. Suppose that $\Sigma$ and $\tilde{\Sigma}$ are smooth, connected, properly embedded self-shrinkers in $\Real^{n+1}\setminus B_{R_0}$ with their boundaries in $S_{R_0}$. If $\Sigma$ and $\tilde{\Sigma}$ are asymptotic to the same cone $C$, i.e. $\lam\Sigma$ and $\lam\tilde{\Sigma}$ converge to $C$ locally smoothly as $\lam\To 0+$\footnote{More precisely, we mean that, for $\forall R>0$ and $k\in\mathbb{N}$, $\lam\Sigma\cap\left(\bar{B}_R\setminus B_{1/R}\right)$ and $\lam\tilde{\Sigma}\cap\left(\bar{B}_R\setminus B_{1/R}\right)$ converge to $C\cap\left(\bar{B}_R\setminus B_{1/R}\right)$ in the $C^k$ topology, as $\lam\To 0+$.}, then $\Sigma$ coincides with $\tilde{\Sigma}$.
\end{thm}

In $\Real^3$, using the desingularization technique, infinitely many smooth complete properly embedded self-shrinkers with discrete rotational symmetries have been successfully constructed by Kapouleas, Kleene and M\o{}ller, and independently by Nguyen; see the recent preprints \cite{KKM} and \cite{NgX1,NgX2,NgX3}. Moreover, the end of each self-shrinker above is a connected graph over the plane outside some compact set and asymptotic to a regular cone. 

It is interesting to compare Theorem \ref{UniqueThm} with the following well-known conjecture (page 39 of \cite{Il2}):

\begin{conj}
\label{AsympConj}
Let $\Sigma\subset\Real^3$ be a smooth complete embedded self-shrinker with at most quadratic area growth\footnote{These assumptions on self-shrinkers in Conjectures \ref{AsympConj} and \ref{CylRigConj} are implied from the content of Lecture 3 in \cite{Il2}.}. Then there exists $R>0$ such that $\Sigma\setminus B_R$ decomposes into a finite number of ends $U_j$ and for each $j$, either
\begin{itemize}
\item[(a)] As $\lam\To 0+$, $\lam U_j$ converges locally smoothly to a cone $C_j$ such that $C_j\setminus\{O\}$ is smooth.
\item[(b)] There is a vector $v_j$ such that, as $\tau\To +\infty$, $U_j-\tau v_j$ converges to the cylinder $\left\{x:\emph{dist}(x,\emph{span}(v_j))=\sqrt{2}\right\}$.
\end{itemize} 
\end{conj}

All known examples of smooth complete embedded self-shrinkers and a lot of numerical evidence suggest Conjecture \ref{AsympConj} to be true; see \cite{KKM}, \cite{NgX1,NgX2,NgX3} and \cite{Ch}. Moreover, another conjecture on rigidity of the self-shrinking cylinder (see page 39 of \cite{Il2}) states that,

\begin{conj}
\label{CylRigConj}
Let $\Sigma\subset\Real^3$ be a smooth complete embedded self-shrinker with at most quadratic area growth. If one end of $\Sigma$ is asymptotic to a cylinder (see (b) of Conjecture \ref{AsympConj}), then $\Sigma$ is isometric to the self-shrinking cylinder.  
\end{conj}

The polynomial volume growth condition on self-shrinkers arises naturally from the analysis of asymptotic behaviour for singularities of the mean curvature flow; see \cite{H1,H2}, \cite{Il2}, \cite{Ec} and \cite{CM2,CM3}. Recently, an equivalence has been shown between Euclidean volume growth and properness of an smooth complete immersed self-shrinker in Euclidean space; see \cite{DX} and \cite{CZh}. Thus it is natural to ask that, given a regular cone $C\subset\Real^3$ with vertex at $O$, how many smooth complete properly embedded self-shrinkers there exist having an end asymptotic to $C$ in $\Real^3$ (see (a) of Conjecture \ref{AsympConj}). Theorem \ref{UniqueThm} gives an upper bound, which is one, to this question. Furthermore, in the sequel paper \cite{W2}, we establish two anisotropic Carleman inequalities and verify Conjecture \ref{CylRigConj} under additional conditions on the rate of convergence at infinity.

Many numerical examples in \cite{Ch} indicate that it is very difficult to classify all the smooth complete embedded self-shrinkers in $\Real^{n+1}$. However, under certain conditions, several results on the classification for self-shrinkers have been obtained since the 1980s. For $n=1$, Abresch and Langer, \cite{AL}, had already shown that the circle is the only simple closed self-shrinking curve. For higher dimensions, Colding and Minicozzi, \cite{CM2}, proved that the only smooth complete embedded self-shrinkers in $\Real^{n+1}$ that are mean convex and have polynomial volume growth are $S^k(\sqrt{2k})\times\Real^{n-k}$, $0\le k\le n$, which generalized an earlier result of Huisken, \cite{H1,H2}. Here $S^k(\sqrt{2k})$ denotes the $k$-dimensional round sphere centered at the origin with radius $\sqrt{2k}$. Moreover, in the same paper, they showed that those self-shrinkers are the only entropy stable ones under the mean curvature flow. Also, in \cite{W1}, we established a Bernstein type theorem for smooth self-shrinkers in $\Real^{n+1}$, which generalized a result of Ecker and Huisken, \cite{EH1}.  Besides, by imposing symmetries, Kleene and M\o{}ller, \cite{KM}, classified the smooth complete embedded self-shrinking hypersurfaces of revolution in $\Real^{n+1}$. In contrast to the previously mentioned results, the feature of Theorem \ref{UniqueThm} is that we do not require either information on self-shrinkers inside compact sets or any assumption on symmetries of self-shrinkers.  

We give a sketch of the proof of Theorem \ref{UniqueThm} in the following. First, outside some compact set $K\subset\Real^{n+1}$, we write $\tilde{\Sigma}\setminus K$ as the graph of a function $v$, which is defined over $\Sigma$ and vanishes at infinity with a certain rate. Next, we derive the differential equation for the function $v$ (see (\ref{JacobiEqn}) in section 2) which involves the Ornstein-Uhlenbeck type operator. One main difficulty comes from that, when changing infinity to the origin, the resulted equation is elliptic but highly singular and degenerate at the origin. Besides, under our assumptions, $v$ need not vanish to infinite order asymptotically. Thus, we cannot apply the strong unique continuation theorems in \cite{GL1,GL2}, \cite{J} or \cite{JK} to conclude that $v$ is identically zero. Instead, we consider a new function $w$, which is defined by a suitable scaling of $v$ on a domain of the space-time, and derive the differential equation for $w$ (see (\ref{HeatEqn}) in section 3). Finally, we show the backward uniqueness for the parabolic equation for $w$. The point of our backward uniqueness result is that the values of $w$ at the parabolic boundary of the domain are not controlled by the assumptions. The idea of the proof is borrowed from the paper \cite {ESS}, in which Escauriaza, Seregin and \v{S}ver\'{a}k proved a similar backward uniqueness for parabolic equations, which are linear perturbations of the standard heat equation in Euclidean space. As applications, they settled a long-standing question concerning sufficient conditions for regularity of solutions for the Navier-Stokes equations in $\Real^3$; see \cite{ESS}. For more bibliographic information on the backward uniqueness for parabolic equations, we recommend the readers refer to \cite{EF}, \cite{EKPV}, \cite{ESS,ESS1}, \cite{K1,K2}, \cite{LS}, \cite{L}, \cite{MZ}, \cite{NgT}, \cite{P}, \cite{Ru}, \cite{SS} and \cite{So}.

There are several interesting applications of our uniqueness theorem. For instance, by Theorem \ref{UniqueThm}, it is not hard to show that not every regular cone with vertex at $O$ has a smooth complete properly embedded self-shrinker asymptotic to it. Consider a rotationally symmetric regular cone $C\subset\Real^{n+1}$ with vertex at $O$. In \cite{KM}, Kleene and M\o{}ller constructed a smooth embedded self-shrinking end of revolution in $\Real^{n+1}$ that is asymptotic to $C$. Thus, Theorem \ref{UniqueThm} implies that any smooth, connected, properly embedded self-shrinker that is asymptotic to $C$ must have a rotational symmetry. Hence, using the classification result for smooth complete embedded self-shrinkers with a rotational symmetry (see \cite{KM}), we conclude that

\begin{cor}
\label{NonexistCor}
Let $C\subset\Real^{n+1}$ be a rotationally symmetric regular cone with vertex at $O$. Assume that $C$ is not a hyperplane. Then there do not exist smooth complete properly embedded self-shrinkers in $\Real^{n+1}$ with an end asymptotic to $C$.  
\end{cor} 

On the other hand, by Huisken's monotonicity formula (see \cite{H1,H2}), self-shrinkers are hypersurfaces in the Euclidean space that are minimal with respect to the Gaussian conformally changed metric; see \cite{A} and \cite{CM2,CM3}. It is well-known that one can construct many smooth complete properly embedded minimal surfaces with asymptotically planar ends, e.g. Costa-Hoffman-Meeks minimal surfaces (see \cite{Co} and \cite{HM}). However, in contrast to the minimal surfaces theory, it follows from Theorem \ref{UniqueThm} that the only smooth complete properly embedded self-shrinkers with ends asymptotic to hyperplanes are hyperplanes. This gives an explanation of the dramatic change of asymptotics of the non-compact ends occurring in the desingularization construction of self-shrinkers in \cite{KKM} and \cite{NgX1,NgX2,NgX3}\footnote{In those papers, one starts with surfaces by desingularizing the intersection of a sphere and a plane, but the resulted self-similar surfaces are asymptotic to cones (not planes). This phenomenon does not happen in the desingularization construction of minimal surfaces.}.

Also, assuming that Conjecture \ref{AsympConj} is true, a gap theorem for self-shrinkers can be established using Theorem \ref{UniqueThm}.

\begin{cor}
\label{GapCor}
Let $\Sigma\subset\Real^3$ be a smooth complete properly embedded self-shrinker. Assume that one end of $\Sigma$ is asymptotic to a regular cone and Conjecture \ref{AsympConj} is true. Then any smooth complete properly embedded self-shrinker in $\Real^3$ that has sublinear growth of Hausdorff distance from $\Sigma$ must coincide with $\Sigma$. 
\end{cor}

\textbf{Acknowledgement.} The author is indebted to Brett Kotschwar for bringing the paper \cite{ESS} to her attention and many inspiring discussions which stimulated the present paper. The author is also very grateful to Tom Ilmanen for many useful suggestions which led to the proof of Lemma \ref{PertLem}. Finally, the author would like to thank Jacob Bernstein and Tobias Colding for their interest in this work and many constructive comments on a preliminary draft. 

\section{Notations and Auxiliary Lemmas}

In this section, we set up the notations for this paper and prove several auxiliary lemmas. In particular, we show that, outside some compact set, $\tilde{\Sigma}$ can be written as the graph of a function $v$ over $\Sigma$ dacaying with a certain rate, and derive the differential equation for $v$. 

Throughout, for any multi-index $\alpha=\left(\alpha_1,\dots,\alpha_{\abs{\alpha}}\right)$, $\partial^{\abs{\alpha}}_\alpha$ denotes the $\abs{\alpha}$-th order partial derivative with respect to the $\alpha_i$-th, $1\le i\le\abs{\alpha}$, coordinates of Euclidean space; the meaning of $\partial_t$ may vary in lemmas and propositions, and we will clarify it in the content; $\nabla$, $\nabla^2$ and $\Delta$ denote the gradient, Hessian and Laplacian on the hypersurfaces appearing in the subscripts respectively; $A$ and $\nabla^i A$ are the second fundamental form and its $i$th covariant derivative respectively; constants in lemmas and propositions depend only on $n$, $\Sigma$, $\tilde{\Sigma}$ and $C$, unless it is specified; constants in the proofs are not preserved when crossing lemmas and propositions. 

Denote $\lam\Sigma\cap\left(\bar{B}_R\setminus B_{1/R}\right)$ by $\Sigma_{R,\lam}$ and $C\cap\left(\bar{B}_R\setminus B_{1/R}\right)$ by $C_R$. We recall that $\Sigma_{R,\lam}$ converges to $C_R$ in the $C^k$ topology as $\lam\To 0+$, if $\Sigma_{R,\lam}$ converges to $C_R$ in the Hausdorff topology, and for any $x\in C_R$ and $\lam>0$ small, $\Sigma_{R,\lam}$ locally (near $x$) is a graph over the tangent hyperplane $T_x C_R$ and the graph of $\lam_{R,\lam}$ converges to the graph of $C_R$ in the usual $C^k$ topology. 

Since $\Sigma$ is a self-shrinker under the mean curvature flow, $\left\{\Sigma_t=\sqrt{t}\ \Sigma\right\}_{t\in (0,1]}$ is a solution to the backward mean curvature flow, i.e.
\begin{equation}
\label{MCFEqn}
\partial_t x=H\textbf{n},
\end{equation}
for $x\in\Sigma_t$ and $t\in (0,1]$ (parametrizing $\Sigma_t$ in a suitable way). In other words, $\partial_t x$ stands for the normal velocity of the hypersurface.
We recommend the readers refer to Chapter 2 of \cite{Ec} for more details and other equivalent definitions. We begin with the following elementary lemma on the geometry of $\Sigma_t$.

\begin{lem}
\label{GeomLem}
There exist $C_1>0$ and $R_1\ge R_0$ such that for $x\in\Sigma_t\setminus B_{R_1}$, $t\in (0,1]$ and $0\le i\le 2$,
\begin{equation}
\label{GeomEqn}
\abs{\nabla^iA(x)}\le C_1\abs{x}^{-i-1}.
\end{equation}
\end{lem}

\begin{proof}
By (\ref{RegConeEqn}) and the fact that $\Gamma$ is a smooth closed embedded submanifold of $S_1$, the second fundamental form and all its covariant derivatives of $C$ are bounded inside the annulus $\bar{B}_2\setminus B_{1/2}$. Furthermore, since $\lam\Sigma\To C$ locally smoothly as $\lam\To 0+$, there exist $\delta_1,\lam_1>0$ such that: if $0<\lam<\delta_1$, then for $0\le i\le 2$, $\abs{\nabla^iA}\le \lam_1$ on $\lam\Sigma\cap\left(\bar{B}_2\setminus B_{1/2}\right)$. Thus, for $x\in\Sigma$ with $\abs{x}\ge R=2\max\{R_0, 1/\delta_1\}$, we can choose $\lam=1/\abs{x}$ such that for $0\le i\le 2$,
\begin{equation}
\label{GeomEqn1}
\abs{\nabla^iA(x)}=\lam^{i+1}\abs{\nabla^i
A(\lam x)}\le\lam_1\abs{x}^{-i-1},
\end{equation}
where note that $\lam x\in\lam\Sigma$ is inside $\bar{B}_2\setminus B_{1/2}$.

Since $\Sigma_t=\sqrt{t}\ \Sigma$, we conclude that for $x\in\Sigma_t\setminus B_R$ and $0\le i\le 2$,
\begin{equation}
\label{GeomEqn2}
\abs{\nabla^iA(x)}=t^{\frac{-i-1}{2}}
\abs{\nabla^iA\left(\frac{x}{\sqrt{t}}\right)}\le\lam_1\abs{x}^{-i-1}.
\end{equation}
\end{proof}

Next, it follows from the assumption of Theorem \ref{UniqueThm} that, outside a compact set, $\Sigma_t$ is given by the graph of a smooth function over $C$. On the other hand, using the proof of Lemma 2.2 on page 30 of \cite{CM1}, we can write $\Sigma_t$ locally as the graph of a smooth function over a fixed hyperplane. 

\begin{lem}
\label{PertLem}
There exist $R_2>R_1$ and compact sets $K_t\subset\Sigma_t$, $0<t\le 1$, such that: $K_t\subset \Sigma_t\cap B_{2R_2}$, and $\Sigma_t\setminus K_t$ is given by the graph of a smooth function $U(\cdot,t):C\setminus \bar{B}_{R_2}\To\Real$. Moreover, there exist $0<\eps_0<1$ and $C_2>0$ such that for $z_0\in C\setminus B_{2R_2}$ and $t\in (0,1]$, the component of $\Sigma_t\cap B_{\eps_0\abs{z_0}}(z_0)$ containing $z_0+U(z_0,t)\textbf{n}(z_0)$ can be written as the graph of a smooth function $u(\cdot,t)$ over the tangent hyperplane $T_{z_0}C$ of $C$ at $z_0$ satisfying that, for $i=0,1,2$,
\begin{equation}
\label{PertEqn}
\abs{D^{i+1}u}\le C_2\abs{z_0}^{-i}\quad\text{and}\quad\abs{D^i\partial_t u}\le C_2\abs{z_0}^{-1-i}. 
\end{equation}
Here $D$ and $D^2$ are the Euclidean gradient and Hessian on $\Real^n$ respectively, and $\partial_t$ denotes the partial derivative with respect to $t$ fixing points in $T_{z_0}C$.
\end{lem}

\begin{proof}
Given $\delta>0$, by the assumption of Theorem \ref{UniqueThm}, there exists $t_0=t_0(\delta)\in (0,R_1^{-2}]$ such that: if $0<t\le t_0$, then $\Sigma_t\cap B_3\setminus B_{1/3}$ can be written as the graph of a function $V(\cdot,t):\Omega_t\To\Real$ such that $\norm{V(\cdot,t)}_{C^1}\le\delta$. Here, the domain $\Omega_t$ of $C$ satisfies that $B_2\setminus B_{1/2}\subset\Omega_t\subset B_4\setminus B_{1/4}$ and $\Omega_t\To C\cap\left(B_3\setminus B_{1/3}\right)$ as $t\To 0$. Thus, we can choose $\delta$ sufficiently small, depending on $C$ inside the annulus $\bar{B}_4\setminus B_{1/4}$, such that for $z\in C\cap\left(B_2\setminus B_{1/2}\right)$, $x_t=z+V(z,t)\textbf{n}(z)\in\Sigma_t$ and $0<t\le t_0$, the distance $\text{dist}(x_t,C)$ from $x_t$ to $C$, which is achieved uniquely at $z$ and equal to $\abs{V(z,t)}$, is less than $1/100$, and $\la\textbf{n}(x_t),\textbf{n}(z)\ra>99/100$. Hence, if $x\in\Sigma_t$ with $\abs{x}^2>t_0^{-1}$ and $0<t\le 1$, then, by the homogeneity of $C$, $\text{dist}(x,C)<\abs{x}/100$, which is attained at a unique point $z$ on $C$, and $\la\textbf{n}(x),\textbf{n}(z)\ra>99/100$. This implies that the nearest point projection $\Pi_t:\Sigma_t\setminus B_{2/\sqrt{t_0}}\To C$ is well-defined for each $t\in (0,1]$ and $\abs{x}/2<\abs{\Pi_t(x)}<2\abs{x}$ for $x\in\Sigma_t\setminus B_{2/\sqrt{t_0}}$. Moreover, the map $\Pi_t$ is injective and the image of $\Pi_t$ contains $C\setminus \bar{B}_{4/\sqrt{t_0}}$. This proves the first part of Lemma \ref{PertLem} with $R_2=4/\sqrt{t_0}$ and $K_t=\Sigma_t\setminus\Pi_t^{-1}(C\setminus\bar{B}_{4/\sqrt{t_0}})$.

It follows from the discussion in the previous paragraph that, given $\delta>0$, there exists $r=r(\delta)>0$ such that for $z_0\in C\setminus B_r$ and $0<t\le 1$, $\abs{\Pi_t^{-1}(z_0)-z_0}<\delta\abs{z_0}$. Thus, by the proof of Lemma 2.2 in \cite{CM1} and Lemma \ref{GeomLem}, there exist $\delta_1\in (0,1)$ and $r_1,\lam_1>0$, depending only on $C_1$, such that: for $z_0\in C\setminus B_{r_1}$ and $t\in (0,1]$, the component of $\Sigma_t\cap B_{\delta_1\abs{z_0}}(z_0)$ containing $\Pi_t^{-1}(z_0)$ can be written as the graph of a smooth function $u(\cdot,t)$ over the tangent hyperplane of $C$ at $z_0$ and $u(\cdot,t)$ satisfies that $\abs{Du(\cdot,t)}\le\lam_1$ and $\abs{D^2u(\cdot,t)}\le\lam_1\abs{z_0}^{-1}$. In fact, Lemma \ref{GeomLem} implies that $\abs{D^3u(\cdot,t)}\le\lam_2\abs{z_0}^{-2}$ and $\abs{D^4u(\cdot,t)}\le\lam_2\abs{z_0}^{-3}$, where $\lam_2>0$ depends only on $C_1$ and $\lam_1$. Moreover, since $\Sigma_t$ varies smoothly with respect to $t$, the function $U$ in this lemma depends on $t$ smoothly and so does $u$. In the following, we parametrize $T_{z_0}C$ by $F:\Real^n\To T_{z_0}C$, $F(p)=z_0+\sum_ip_ie_i$, where $p=(p_1,\dots,p_n)$ and $\{e_1,\dots,e_n\}$ is an orthonormal basis of $T_{z_0}C-z_0$. And we identify $u(p,t)$ with $u(F(p),t)$. Note that $\partial_t$ means the partial derivative with respect to $t$ fixing $p$. Since $\{\Sigma_t\}_{t\in (0,1]}$ moves by mean curvature backward, we derive the differential equation for $u$:
\begin{equation}
\label{GraphMCFEqn}
-\partial_t u=\sqrt{1+\abs{Du}^2}\ \text{div}\left(\frac{Du}{\sqrt{1+\abs{Du}^2}}\right).
\end{equation}
Thus, $\abs{\partial_t u}\le\lam_3\abs{z_0}^{-1}$, where $\lam_3>0$ depends only on $\lam_1$. Next, differentiating once with respect to the $i$th coordinate on both sides of the equation (\ref{GraphMCFEqn}) gives
\begin{equation*}
\label{GraphMCFEqn1}
-\partial_i\partial_t u=\frac{\sum_k\partial_k u\partial^2_{ik} u}{\sqrt{1+\abs{Du}^2}}\ \text{div}\left(\frac{Du}{\sqrt{1+\abs{Du}^2}}\right)
+\sqrt{1+\abs{Du}^2}\ \partial_i\text{div}\left(\frac{Du}{\sqrt{1+\abs{Du}^2}}\right).
\end{equation*}
Note that the divergence term on the right hand side of the equation (\ref{GraphMCFEqn}) is the mean curvature of $\Sigma_t$ at $\text{Graph}(u)$. Thus, Lemma \ref{GeomLem} implies that
\begin{equation}
\label{PertEqn1}
\abs{\partial_i\text{div}\left(\frac{Du}{\sqrt{1+\abs{Du}^2}}\right)}
\le\lam_4\abs{z_0}^{-2},
\end{equation}
where $\lam_4>0$ depends only on $\lam_1$ and $C_1$. Hence, there exists $\lam_5>0$, depending only on $\lam_1$, $\lam_2$ and $\lam_4$, such that $\abs{\partial_i\partial_t u}\le\lam_5\abs{z_0}^{-2}$. Similarly, differentiating twice with respect to the $i$th and $j$th coordinates on the both sides of the equation (\ref{GraphMCFEqn}), we get
\begin{equation*}
\label{GraphMCFEqn2}
\begin{split}
-\partial^2_{ij}\partial_t u
=&\left(\frac{\sum_k\partial^2_{ik}u\partial^2_{jk}u+\partial_k u\partial^3_{ijk}u}{\sqrt{1+\abs{Du}^2}}-\frac{\sum_{k,l}\partial_k u\partial_l u\partial^2_{ik}u\partial^2_{jl}u}{(1+\abs{Du}^2)^{3/2}}\right)
\text{div}\left(\frac{Du}{\sqrt{1+\abs{Du}^2}}\right)\\
&+\frac{\sum_k\partial_k u\partial^2_{ik}u}{\sqrt{1+\abs{Du}^2}}
\ \partial_j\text{div}\left(\frac{Du}{\sqrt{1+\abs{Du}^2}}\right)+\frac{\sum_k\partial_k u\partial^2_{jk}u}{\sqrt{1+\abs{Du}^2}}
\ \partial_i\text{div}\left(\frac{Du}{\sqrt{1+\abs{Du}^2}}\right)\\
&+\sqrt{1+\abs{Du}^2}\ \partial^2_{ij}\text{div}\left(\frac{Du}{\sqrt{1+\abs{Du}^2}}\right).
\end{split}
\end{equation*}
Note that, by Lemma \ref{GeomLem}, we have that
\begin{equation}
\label{PertEqn2}
\abs{\partial^2_{ij}\text{div}\left(\frac{Du}{\sqrt{1+\abs{Du}^2}}\right)}\le \lam_6\abs{z_0}^{-3},
\end{equation}
where $\lam_6>0$ depends on $\lam_1$ and $\lam_2$.
Hence, there exists $\lam_7>0$, depending only on $\lam_1$, $\lam_2$, $\lam_4$ and $\lam_6$, such that $\abs{\partial^2_{ij}\partial_t u}\le\lam_7\abs{z_0}^{-3}$.
\end{proof}

Thus, by Lemmas \ref{GeomLem} and \ref{PertLem}, we can write $\tilde{\Sigma}$ outside a compact set as the graph of a function over $\Sigma$. Namely,

\begin{lem}
\label{GraphLem}
There exist $R_3>R_2$ and $C_3>0$ such that, outside a compact set,  $\tilde{\Sigma}$ is given by the graph of a smooth function $v:\Sigma\setminus \bar{B}_{R_3}\To\Real$ satisfying that for $x\in\Sigma\setminus \bar{B}_{R_3}$ and $0\le i\le 2$,
\begin{equation}
 \label{GraphEqn}
\abs{\nabla^i_\Sigma v(x)}\le C_3\abs{x}^{-1-i}.
\end{equation}
\end{lem}

\begin{proof}
First, using the same argument in Lemma \ref{PertLem}, there exist $\tilde{R}_2>R_1$ and compact sets $\tilde{K}_t\subset\tilde{\Sigma}_t$, $0<t\le 1$, such that: $\tilde{K}_t\subset B_{2\tilde{R}_2}$, and $\tilde{\Sigma}_t\setminus \tilde{K}_t$ is given by the graph of a smooth function $\tilde{U}(\cdot,t):C\setminus \bar{B}_{\tilde{R}_2}\To\Real$. Moreover, there exist $0<\tilde{\eps}_0<1$ and $\tilde{C}_2>0$ such that for $z_0\in C\setminus B_{2\tilde{R}_2}$ and $t\in (0,1]$, the component of $\tilde{\Sigma}_t\cap B_{\tilde{\eps}_0\abs{z_0}}(z_0)$ containing $z_0+\tilde{U}(z_0,t)\textbf{n}(z_0)$ can be written as the graph of a smooth function $\tilde{u}(\cdot,t)$ over the tangent hyperplane of $C$ at $z_0$ satisfying that, for $i=0,1,2$,
\begin{equation}
\label{GraphEqn1}
\abs{D^{i+1}\tilde{u}}\le \tilde{C}_2\abs{z_0}^{-i}\quad\text{and}\quad\abs{D^i\partial_t\tilde{u}}\le \tilde{C}_2\abs{z_0}^{-1-i}. 
\end{equation}
For $r_1>2\max\{R_2,\tilde{R}_2,\sqrt{C_2},\sqrt{\tilde{C}_2} \}$, let $\Pi:\Sigma\setminus B_{r_1}\To C$ and $\tilde{\Pi}:\tilde{\Sigma}\setminus B_{r_1}\To C$ be the nearest point projections. Thus, if $y\in\tilde{\Sigma}\setminus B_{4r_1}$, then $z=\tilde{\Pi}(y)\in C\setminus B_{2r_1}$ and $\text{dist}(y,\Sigma)\le\abs{y-z}+\abs{z-\Pi^{-1}(z)}\le 2(C_2+\tilde{C}_2)\abs{y}^{-1}$. Hence, by Lemma \ref{GeomLem}, we can choose $r_1\gg 1$, depending only on $C_1$, $C_2$ and $\tilde{C}_2$, such that for $y\in\tilde{\Sigma}\setminus B_{4r_1}$, there exists a unique $x\in\Sigma$ satisfying that $\abs{y-x}=\text{dist}(y,\Sigma)$.

Next, for $x\in\Sigma\setminus B_{2r_1}$, Lemma \ref{PertLem} implies that $z'=\Pi(x)\in C\setminus B_{r_1}$, $\abs{x-z'}\le C_2\abs{z'}^{-1}$ and $\la\textbf{n}(x),\textbf{n}(z')\ra>1-\mu_1\abs{z'}^{-4}$, where $\mu_1>0$ depends on $C_2$. Thus, if $r_1$ is sufficiently large, depending only on $C_2$ and $\tilde{C}_2$, then the component of $\tilde{\Sigma}\cap B_{\tilde{\eps}_0\abs{z'}/2}(z')$ containing $\tilde{\Pi}^{-1}(z')$ can be written as the graph of the function $f$ over the tangent hyperplane of $\Sigma$ at $x$. And $\abs{f(x)}\le\mu_2\abs{z'}^{-1}$ for some $\mu_2$, depending on $C_2$ and $\tilde{C}_2$. Hence, there exists $r_2>2r_1$ large, depending only on $\mu_2$ and $C_1$, such that the image of the nearest point projection from $\tilde{\Sigma}\setminus B_{4r_1}\To\Sigma$ contains $\Sigma\setminus B_{r_2}$.

Finally, it is easy to show that there exists $\delta_1>0$, depending only on $C$, such that: if $z_1,z_2\in C\setminus B_{r_1}$ and $\abs{z_1-z_2}<\delta_1$, then the geodesic distance between $z_1$ and $z_2$ on $C$, $\text{dist}_C(z_1,z_2)$, is less than $2\abs{z_1-z_2}$. Note that for $y\in\tilde{\Sigma}\setminus B_{2r_1}$, $\tilde{\Pi}(y)\in C\setminus B_{r_1}$, $|\tilde{U}(\tilde{\Pi}(y),1)|\le \tilde{C}_2/|\tilde{\Pi}(y)|$ and $|\nabla_C\tilde{U}(\cdot,1)|$ is uniformly small on $C\setminus B_{r_1}$. Thus, there exist $\delta_2>0$, $r_3\gg 1$ and $\mu_3>1$, depending only on $\tilde{C}_2$, $\delta_1$ and $r_1$, such that: if $y_1$, $y_2\in\tilde{\Sigma}\setminus B_{r_3}$ and $\abs{y_1-y_2}\le\delta_2$, then the geodesic distance $\text{dist}_{\tilde{\Sigma}}(y_1,y_2)\le\mu_3\abs{y_1-y_2}$.  
Hence, by the discussion in the previous paragraph, the nearest point projection from $\tilde{\Sigma}\setminus B_{r_4}$ to $\Sigma$ is injective, for some large $r_4>4\max\{r_1,r_3\}$, depending on $C_2$ and $\tilde{C}_2$. Therefore, choosing $r_5=2\max\{r_2,r_4\}$, we conclude that, outside a compact set, $\tilde{\Sigma}$ is given by the graph of a smooth function $v$ over $\Sigma\setminus\bar{B}_{r_5}$. 

In the following, we derive the bounds of $v$ and its derivatives up to the second order. Fix $x$ in $\Sigma\setminus\bar{B}_{r_5}$. It is clear that $\abs{v(x)}\le\lam_1\abs{x}^{-1}$, where $\lam_1>0$ depends only on $C_2$ and $\tilde{C}_2$. Let $z'=\Pi(x)$ and $y=x+v(x)\textbf{n}(x)$. Thus, by Lemma \ref{PertLem}, $\abs{x-z'}$ and $|y-\tilde{\Pi}^{-1}(z')|$ are bounded by multiples of $1/\abs{x}$. Hence, the geodesic distance $\text{dist}_{\tilde{\Sigma}}(y,\tilde{\Pi}^{-1}(z'))$ is also a multiple of $1/\abs{x}$. Therefore, we can write the component of $\Sigma\cap B_{\eps_0\abs{z}}(z')$ containing $x$ and the component of $\tilde{\Sigma}\cap B_{\tilde{\eps}_0\abs{z}}(z')$ containing $y$ as the graphs of smooth functions $u_1=u(\cdot,1)$ and $\tilde{u}_1=\tilde{u}(\cdot,1)$ over the tangent hyperplane $T_{z'}C$ of $C$ at $z'$ respectively. We parametrize $T_{z'}C$ by $F:\Real^n\To T_{z'}C$, $F(p)=z'+\sum_ip_ie_i$, where $p=(p_1,\dots,p_n)$ and $\{e_1,\dots,e_n\}$ is an orthonormal basis of $T_{z'}C-z'$. For $p,q\in\Real^n$, we identify $u_1(p)$, $\tilde{u}_1(q)$ and $v(p)$ with $u_1(F(p))$, $\tilde{u}_1(F(q))$ and $v(F(p))$ respectively. Thus, for $h=1,2\dots,n$,
\begin{align}
\label{GraphEqn2.1}
q_h & = p_h-\frac{\partial_hu_1}{\sqrt{1+\abs{Du_1}^2}}\ v,\\
\label{GraphEqn2.3}
\tilde{u}_1(q) & = u_1(p)+\frac{1}{\sqrt{1+\abs{Du_1}^2}}\ v.
\end{align}
Differentiating the equations (\ref{GraphEqn2.1}) and (\ref{GraphEqn2.3}) with respect to $p_i$ gives
\begin{equation}
\label{GraphEqn3.1}
\begin{split}
\partial_i q_h
= &\delta_{hi}-\frac{\partial^2_{hi}u_1}{\sqrt{1+\abs{Du_1}^2}}\  v+\frac{\partial_hu_1}{(1+\abs{Du_1}^2)^{3/2}}\left( \sum_l\partial_lu_1\partial^2_{li}u_1\right) v\\
&-\frac{\partial_h u_1}{\sqrt{1+\abs{Du_1}^2}}\ \partial_iv,\\
\end{split}
\end{equation}
\begin{equation}
\label{GraphEqn3.3}
\begin{split}
\sum_k\partial_k\tilde{u}_1\partial_i q_k
=&\partial_i u_1+\frac{\partial_i v}{\sqrt{1+\abs{Du_1}^2}}-\frac{v}{(1+\abs{Du_1}^2)^{3/2}}\ \sum_l\partial_lu_1\partial^2_{li}u_1.
\end{split}
\end{equation}
Furthermore, plugging the equation (\ref{GraphEqn3.1}) into the equation (\ref{GraphEqn3.3}) gives
\begin{equation}
\label{GraphEqn4}
\begin{split}
\frac{1+\sum_k\partial_k\tilde{u}_1\partial_ku_1}{\sqrt{1+\abs{Du_1}^2}}\  \partial_iv
=&(\partial_i\tilde{u}_1-\partial_iu_1)-\frac{\sum_k
\partial_k\tilde{u}_1\partial^2_{ki}u_1}{\sqrt{1+\abs{Du_1}^2}}\ v\\
&+\frac{1+\sum_k\partial_k\tilde{u}_1\partial_ku_1}{(1+\abs{Du_1}^2)^{3/2}}
\left(\sum_l\partial_lu_1\partial^2_{li}u_1\right)v.
\end{split}
\end{equation} 
Note that $\abs{q(0)}\le \lam_1\abs{x}^{-1}$ and $\abs{x}$, $\abs{y}$ and $\abs{z'}$ are comparable. Thus, by the assumption of Theorem \ref{UniqueThm}, (\ref{PertEqn}) and (\ref{GraphEqn1}), there exists $\lam_2>0$, depending only on $\lam_1$, $C_2$ and $\tilde{C}_2$ such that
\begin{equation}
\label{GraphEqn5}
\begin{split}
&\abs{\partial_i\tilde{u}_1(q(0))-\partial_i u_1(0)}\\
\le&\abs{\partial_i\tilde{u}_1(q(0))-\partial_i \tilde{u}_1(0)}
+\abs{\partial_i \tilde{u}_1(0)-\partial_i u_1(0)}\\
\le &\lam_2\abs{x}^{-2}.
\end{split}
\end{equation}
Hence, it follows from the upper bound of $v$, (\ref{PertEqn}), (\ref{GraphEqn1}) and (\ref{GraphEqn5}) that $\abs{\partial_iv(0)}\le \lam_3\abs{x}^{-2}$ for some $\lam_3>0$ depending only on $\lam_2$, $C_2$ and $\tilde{C}_2$ and thus the same holds true for $\abs{\nabla_\Sigma v}$ at $x$.

Next, differentiating the equations (\ref{GraphEqn3.1}) and (\ref{GraphEqn3.3}) with respect to $p_j$ gives that
\begin{equation}
\label{GraphEqn6.1}
\begin{split}
\partial^2_{ij} q_h= &-\frac{\partial^3_{hij} u_1}{\sqrt{1+\abs{Du_1}^2}}\  v+\frac{\partial^2_{hi}u_1}{(1+\abs{Du_1}^2)^{3/2}}\left(\sum_l \partial_lu_1\partial^2_{lj}u_1\right)v\\
&-\frac{3\partial_hu_1}{(1+\abs{Du_1}^2)^{5/2}}\left(\sum_l \partial_lu_1\partial^2_{li}u_1\right)\left(\sum_m \partial_m u_1\partial^2_{mj}u_1\right)v\\
&+\frac{\partial_hu_1}{(1+\abs{Du_1}^2)^{3/2}}\left(\sum_l\partial^2_{li}u_1
\partial^2_{lj}u_1+\partial_lu_1\partial^3_{lij}u_1\right)v\\
&+\frac{\partial^2_{hj}u_1}{(1+\abs{Du_1}^2)^{3/2}}\left(\sum_l \partial_lu_1\partial^2_{li}u_1\right)v
-\frac{\partial^2_{hi}u_1\partial_j v+\partial^2_{hj}u_1\partial_i v}{\sqrt{1+\abs{Du_1}^2}}\\
&+\frac{\partial_hu_1}{(1+\abs{Du_1}^2)^{3/2}}\sum_l \partial_lu_1(\partial^2_{li}u_1\partial_j v+\partial^2_{li}u_1\partial_i v)\\
&-\frac{\partial_hu_1}{\sqrt{1+\abs{Du_1}^2}}\ \partial^2_{ij} v,
\end{split}
\end{equation}
\begin{equation}
\label{GraphEqn6.3}
\begin{split}
&\sum_{k,n}\left(\partial^2_{kn}\tilde{u}_1\partial_i q_k \partial_j q_n+\partial_k\tilde{u}_1 \partial^2_{ij} q_k\right)\\
=& \partial^2_{ij}u_1+\frac{3}{(1+\abs{Du_1}^2)^{5/2}}\left(\sum_l\partial_lu_1
\partial^2_{li}u_1\right)\left(\sum_m\partial_m u_1\partial^2_{mj}u_1\right)v\\
&-\frac{1}
{(1+\abs{Du_1}^2)^{3/2}}\left(\sum_l\partial^2_{li}u_1\partial^2_{lj}u_1
+\partial_lu_1\partial^3_{lij}u_1\right)v\\
& -\frac{\sum_l\partial_lu_1\partial^2_{lj}u_1}{(1+\abs{Du_1}^2)^{3/2}}\ \partial_i v
-\frac{\sum_l\partial_lu_1\partial^2_{li}u_1}
{(1+\abs{Du_1}^2)^{3/2}}\ \partial_j v
+\frac{\partial^2_{ij} v}{\sqrt{1+\abs{Du_1}^2}}.
\end{split}
\end{equation}
It follows from Lemma \ref{PertLem} and the bounds for $\abs{v(0)}$ and $\abs{Dv(0)}$  that, for $h=1,2,\dots,n$,
\begin{equation}
\label{GraphEqn7}
\abs{\partial^2_{ij} q_h(0)+\frac{\partial_hu_1(0)}{\sqrt{1+\abs{Du_1(0)}^2}}\  \partial^2_{ij} v(0)}\le\lam_4\abs{x}^{-3},
\end{equation}
where $\lam_4>0$ depends only on $C_2$, $\lam_1$ and $\lam_3$. Thus, by the assumption of Theorem \ref{UniqueThm}, (\ref{PertEqn}), (\ref{GraphEqn1}) and $\abs{q(0)}=\lam_1\abs{x}^{-1}$, we estimate
\begin{equation}
\label{GraphEqn8}
\begin{split}
&\abs{\partial^2_{ij}\tilde{u}_1(q(0))-\partial^2_{ij}u_1(0)}\\
\le&\abs{\partial^2_{ij}\tilde{u}_1(q(0))-\partial^2_{ij}\tilde{u}_1(0)}
+\abs{\partial^2_{ij}\tilde{u}_1(0)-\partial^2_{ij}u_1(0)}\\
\le & \lam_5\abs{x}^{-3}. 
\end{split}
\end{equation}
Here $\lam_5>0$ depends only on $C_2$, $\tilde{C}_2$ and $\lam_1$.
Since $\abs{D\tilde{u}_1(q(0))}\le\tilde{C}_2$, $\abs{D^2\tilde{u}_1(q(0))}\le \tilde{C}_2\abs{x}^{-1}$ and there exists $\lam_6>0$, depending only on $C_2$, $\lam_1$ and $\lam_3$, such that
\begin{equation}
\label{GraphEqn9}
\abs{\partial_i q_h(0)-\delta_{ih}}\le\lam_6\abs{x}^{-2},
\end{equation}
simplifying the equation (\ref{GraphEqn6.3}) gives that
\begin{equation}
\label{GraphEqn10}
\abs{\frac{1+\sum_k\partial_ku_1(0)\partial_k\tilde{u}_1(0)}{\sqrt{1+\abs{Du_1(0)}^2}}\  \partial^2_{ij}v(0)-\left(\partial^2_{ij}\tilde{u}_1(0)-\partial^2_{ij}u_1(0)\right)}\le\lam_7\abs{x}^{-3}.
\end{equation}
Here $\lam_7>0$ depends only on $C_2$, $\tilde{C}_2$, $\lam_1$ and $\lam_3$. Therefore, it follows from (\ref{GraphEqn5}), (\ref{GraphEqn8}), (\ref{GraphEqn10}) and Lemma \ref{PertLem} that $\abs{\nabla^2_\Sigma v(x)}\le\lam_8\abs{x}^{-3}$, where $\lam_8>0$ depends on $C_2$, $\lam_2$, $\lam_5$ and $\lam_7$.
\end{proof}

Finally, we conclude this section by deriving the differential equation for $v$ from the definition of self-shrinkers. Note that the operator $L_\Sigma$ in Lemma \ref{JacobiLem} is a small perturbation of the stability operator introduced in \cite{CM2}. 

\begin{lem}
\label{JacobiLem}
There exists $C_4>0$ such that at $x\in\Sigma\setminus B_{R_3}$,
\begin{equation}
 \label{JacobiEqn}
L_\Sigma v=\Delta_{\Sigma}v-\frac{1}{2}\la x,\nabla_\Sigma v\ra+\left(\abs{A}^2+\frac{1}{2}\right)v+Q\left(x,v,
\nabla_\Sigma v,\nabla^2_\Sigma v\right)=0,
\end{equation}
where the function $Q$ satisfies that
\begin{equation}
 \label{QuadEqn}
\abs{Q\left(x,v,\nabla_\Sigma v,\nabla^2_\Sigma v\right)}\le C_4 \abs{x}^{-2}\left(\abs{v}+\abs{\nabla_\Sigma v}\right).
\end{equation}
\end{lem}

\begin{proof}
Fix $x_0\in\Sigma\setminus B_{R_3}$. We choose a local parametrization of $\Sigma$ in a neighborhood of $x_0$, $F:\Omega\To\Sigma$, satisfying that: $F(0)=x_0$, $\la\partial_iF(0),\partial_jF(0)\ra=\delta_{ij}$ and $\partial^2_{ij}F(0)=a_{ij}(0)\textbf{n}(x_0)$ with $a_{ij}=A(\partial_i F,\partial_j F)$ and $a_{ij}(0)=0$ if $i\ne j$. Here $\Omega$ is a domain in $\Real^n$ containing $0$. Thus, in a neighborhood of $y_0=x_0+v(x_0)\textbf{n}(x_0)$, there exists a local parameterization of $\tilde{\Sigma}$, $\tilde{F}:\Omega\To\tilde{\Sigma}$, such that for $p\in\Omega$,
\begin{equation}
\label{JacobiEqn1}
\tilde{F}(p)=F(p)+v(p)\textbf{n}(p).
\end{equation} 
Here, we identify $v(p)$ and $\textbf{n}(p)$ with $v(F(p))$ and $\textbf{n}(F(p))$ respectively.

First, we calculate the tangent vectors $\partial_i \tilde{F}$ for $1\le i\le n$. Namely,
\begin{equation}
\label{JacobiEqn2}
\partial_i\tilde{F}=\partial_i F+(\partial_i v)\textbf{n}
+v\partial_i\textbf{n}.
\end{equation}
Thus, the unit normal vector $\textbf{n}(\tilde{F}(0))$ to $\tilde{\Sigma}$ at $\tilde{F}(0)$ parallels to the following vector, that is,
\begin{equation}
\label{JacobiEqn4}
\textbf{N}=-\sum_k \left[\prod_{l\ne k}(1-a_{ll}v)\right]\left(\partial_k v\right)\partial_k F
+\left[\prod_k(1-a_{kk}v)\right]\textbf{n}.
\end{equation}

Next, we calculate the second derivatives of $\tilde{F}$ at $p=0$. It follows from (\ref{JacobiEqn2}) that
\begin{equation}
\label{JacobiEqn5}
\partial^2_{ij}\tilde{F}=\partial^2_{ij}F+\left(\partial^2_{ij}v\right)\textbf{n}
+\left(\partial_iv\right)\partial_j\textbf{n}+\left(\partial_jv\right)\partial_i\textbf{n}
+v\partial^2_{ij}\textbf{n}.
\end{equation}
Note that, at $p=0$, 
\begin{equation}
\label{JacobiEqn6}
\begin{split}
\partial^2_{ij}\textbf{n}&=\sum_k\left\la\partial^2_{ij}\textbf{n},\partial_kF\right\ra\partial_kF
+\left\la\partial^2_{ij}\textbf{n},\textbf{n}\right\ra\textbf{n}\\
&=-\sum_k\left(\partial_ja_{ik}\right)\partial_k F-a_{ii}a_{jj}\delta_{ij}\textbf{n}.
\end{split}
\end{equation}
Plugging (\ref{JacobiEqn6}) into (\ref{JacobiEqn5}) gives that, at $p=0$,
\begin{equation}
\label{JacobiEqn7}
\begin{split}
\partial^2_{ij}\tilde{F}=&-a_{ii}\left(\partial_j v\right)\partial_i F-a_{jj}\left(\partial_i v\right)\partial_j F-\sum_k\left(\partial_j a_{ik}\right)v\partial_k F\\
&+\left(a_{ij}-a_{ii}a_{jj}\delta_{ij}v+\partial_{ij}^2v\right)\textbf{n}.
\end{split}
\end{equation}
Thus, making inner product with the vector $\textbf{N}$, we obtain that, at $p=0$,
\begin{equation}
\label{JacobiEqn8}
\begin{split}
\left\la\partial^2_{ij}\tilde{F},\textbf{N}\right\ra=&
a_{ii}\left(\partial_i v\right)\left(\partial_j v\right)\prod_{k\ne i}(1-a_{kk}v)+a_{jj}\left(\partial_i v\right)\left(\partial_j v\right)\prod_{k\ne j}(1-a_{kk}v)\\
&+\left(a_{ij}-a_{ii}a_{jj}\delta_{ij}v+\partial_{ij}^2v\right)\prod_k(1-a_{kk}v)\\
&+v\sum_k\left[\prod_{l\ne k}(1-a_{ll}v)\right]\left(\partial_j a_{ik}\right)\partial_k v.
\end{split}
\end{equation}

Next, it is easy to calculate the pull back metric $g=\left(g_{ij}\right)$ from $\tilde{\Sigma}$, that is, at $p=0$,
\begin{equation}
\label{Metric}
g_{ij}=\left\la\partial_i F,\partial_j F\right\ra=(1-a_{ii}v)(1-a_{jj}v)\delta_{ij}+(\partial_i v)(\partial_j v).
\end{equation}
Thus, at $p=0$, the determinant $\det(g)$ and the inverse $g^{-1}=\left(g^{ij}\right)$ of $g$ are:
\begin{align}
\label{JacobiEqn9.1}
&\det(g)=1-2\sum_k a_{kk}v+Q_{1}(p,v,Dv),\\
\label{JacobiEqn9.2}
&g^{ij}\det(g)=\left\{
\begin{array}{ll}
Q_{2ij}(p,v,Dv)\partial_i v & \text{if}\ i\ne j\\
1-2\sum_{k\ne i}a_{kk}v+Q_{2ij}(p,v,Dv) & \text{if}\ i=j.
\end{array}
\right.
\end{align}
Here, by Lemmas \ref{GeomLem} and \ref{GraphLem}, there exists $\lam_1>0$, depending only on $C_1$ and $C_3$, such that $\abs{Q_{1}(0,v,Dv)}\le\lam_1\abs{x_0}^{-2}\left(\abs{v(0)}+\abs{Dv(0)}\right)$  and so do $\abs{Q_{2ij}(0,v,Dv)}$ for $1\le i,j\le n$.

Finally, we compute, at $p=0$, 
\begin{equation}
\label{JacobiEqn10}
\left\la\tilde{F},\textbf{N}\right\ra=-\sum_k\left[\prod_{l\ne k}(1-a_{ll}v)\right]\left\la F,\partial_k F\right\ra\partial_k v
+\left(\left\la F,\textbf{n}\right\ra+v\right)\prod_k(1-a_{kk}v).
\end{equation}

Since $\tilde{\Sigma}$ is a self-shrinker under the mean curvature flow, we have that
\begin{equation}
\label{JacobiEqn11}
\sum_{i,j}g^{ij}\left\la\partial^2_{ij}\tilde{F},\textbf{N}\right\ra
=-\frac{1}{2}\left\la\tilde{F},\textbf{N}\right\ra.
\end{equation}
Substituting all the previous computation into the equation (\ref{JacobiEqn11}) gives that, at $p=0$,
\begin{equation}
\label{JacobiEqn12}
\sum_k\partial^2_{kk}v-\frac{1}{2}\sum_k\left\la F,\partial_k F\right\ra\partial_k v
+\left(\abs{A}^2+\frac{1}{2}\right)v=\tilde{Q}\left(p,v,Dv,D^2v\right),
\end{equation}
where the function $\tilde{Q}$ satisfies that
\begin{equation}
\label{JacobiEqn13}
\abs{\tilde{Q}\left(0, v,Dv,D^2v\right)}
\le\lam_2\abs{x_0}^{-2}\left(\abs{v(0)}+\abs{Dv(0)}\right).
\end{equation}
Here $\lam_2>0$ depends only on $C_1$ and $C_3$ in Lemma \ref{GraphLem}. Therefore, by the choice of the parameterization in a neighborhood of  $x_0$, Lemma \ref{JacobiLem} follows immediately from (\ref{JacobiEqn13}).
\end{proof}

\section{Proof of Theorem \ref{UniqueThm}}

By the discussion in section 2, Theorem \ref{UniqueThm} can be related to the unique continuation problem for a weakly elliptic equation. Consider the simplest case that $\Sigma=\Real^n\times\{0\}$. We define $\bar{v}(r,\theta)=v(1/r,\theta)$, where $(r,\theta)\in (0,+\infty)\times S^{n-1}$ is the spherical coordinates of $\Real^n\times\{0\}$. Then, the equation (\ref{JacobiEqn}) gives that
\begin{equation}
\label{UniConEqn1}
\Delta_{\Real^n}\bar{v}+\frac{1}{2}\left[r^{-3}+4(2-n)r^{-1}\right]\partial_r\bar{v}
+\frac{1}{2}r^{-4}\bar{v}+\bar{Q}\left(r,\theta,\bar{v},D\bar{v},D^2\bar{v}\right)=0,
\end{equation}
or equivalently,
\begin{equation}
\label{UniConEqn2}
\text{div}\left(r^{4-2n}\text{e}^{-\frac{1}{4r^2}}D\bar{v}\right)
+\frac{1}{2}r^{-2n}\text{e}^{-\frac{1}{4r^2}}\bar{v}
+r^{4-2n}\text{e}^{-\frac{1}{4r^2}}\bar{Q}\left(r,\theta,\bar{v},D\bar{v},D^2\bar{v}\right)=0.
\end{equation}
Here, the funtion $\bar{Q}$ satisfies that
\begin{equation}
\label{UniConEqn3}
\abs{\bar{Q}\left(r,\theta,\bar{v},D\bar{v},D^2\bar{v}\right)}\le C_4r^{-2}\left(\abs{\bar{v}}+r^2\abs{D\bar{v}}\right).
\end{equation}
Hence, we cannot apply some well-known strong unique continuation theorems in \cite{J}, \cite{JK} and \cite{GL1,GL2} to the equations (\ref{UniConEqn1}) or (\ref{UniConEqn2}). In \cite{PW}, Pan and Wolff proved the unique continuation results for operators in the general form, $\Delta_{\Real^n}+\la W(x),\nabla_{\Real^n}\ra+V(x)$, where $\abs{x}^{-2}\abs{V(x)}$ is bounded and $\abs{x}^{-1}\abs{W(x)}$ is small. They also constructed counter examples when the conditions are violated. However, it seems impossible to transform the equation (\ref{UniConEqn2}) to that of their form by changing coordinates. Besides, $\bar{v}$ need not vanish of infinite order at the origin under our assumption. For general $\Sigma$, global coordinates of $\Sigma$ near infinity may not even exist. 

Instead, we associate the elliptic differential equation (\ref{JacobiEqn}) to the parabolic equation in the following lemma. Define a function $w:\cup_{t\in (0,1]}\Sigma_t\times\{t\}\To\Real$ by 
\begin{equation}
\label{ParaEqn}
w(x,t)=\sqrt{t}\ v\left(\frac{x}{\sqrt{t}}\right).
\end{equation}
Let $d/dt$ denote the total derivative with respect to time $t$. Then, by Lemma \ref{JacobiLem}, $w$ satisfies the following equation:

\begin{lem}
\label{HeatLem}
Given $\eps_1>0$, there exists $R>R_3$ such that on $\cup_{t\in (0,1]}\Sigma_t\setminus B_R\times\{t\}$,
\begin{equation}
 \label{HeatEqn}
\abs{\frac{dw}{dt}+\Delta_{\Sigma_t}w}\le \eps_1 \left(\abs{w}+\abs{\nabla_{\Sigma_t}w}\right).
\end{equation}
\end{lem}

\begin{proof}
By a straightforward computation, we get that
\begin{equation}
\label{HeatEqn1}
\begin{split}
&\frac{dw}{dt}+\Delta_{\Sigma_t}w\\
= &\frac{1}{2\sqrt{t}}\ v\left(\frac{x}{\sqrt{t}}\right)+\sqrt{t}\left\la \nabla_\Sigma v\left(\frac{x}{\sqrt{t}}\right),\frac{1}{\sqrt{t}}\frac{\partial x}{\partial t}
-\frac{x}{2t\sqrt{t}}\right\ra
+\frac{1}{\sqrt{t}}\ \Delta_\Sigma v\left(\frac{x}{\sqrt{t}}\right).
\end{split}
\end{equation}
Note that for $x\in\Sigma_t$ and $t\in(0,1]$,
\begin{equation}
\label{HeatEqn2}
\partial_t x=H\textbf{n}\quad\text{and}\quad H=\frac{\la x,\textbf{n}\ra}{2t}.
\end{equation}
Hence, the equation (\ref{HeatEqn1}) and (\ref{HeatEqn2}) give that,
\begin{equation}
\label{HeatEqn3}
\frac{dw}{dt}+\Delta_{\Sigma_t}w=\frac{1}{\sqrt{t}}\left(\Delta_\Sigma v\left(\frac{x}{\sqrt{t}}\right)-\frac{1}{2}\left\la\frac{x}{\sqrt{t}},
\nabla_\Sigma v\left(\frac{x}{\sqrt{t}}\right)\right\ra+\frac{1}{2}
v\left(\frac{x}{\sqrt{t}}\right)\right).
\end{equation}
Therefore, Lemma \ref{HeatLem} follows immediately from Lemmas \ref{GeomLem} and \ref{JacobiLem}.
\end{proof}

The following is devoted to establishing a Carleman inequality by adapting the arguments in \cite{ESS}. First, we prove the following key identity (see (\ref{KeyIdEqn}) below) which is a generalization of that in Lemma 1 of \cite{ESS} to our geometric setting\footnote{Under slightly different notations, a similar identity might have already been implicitly obtained in \cite{EF}.}. Namely, 
\begin{lem}
\label{KeyIdLem}
Assume that $\phi$ and $\Ker$ are smooth functions on $\cup_{t\in (0,1]}\Sigma_t\times\{t\}$ and that $\Ker$ is positive. And let $\mathcal{F}=\left(d\Ker/dt-\Delta_{\Sigma_t}\Ker\right)/\Ker$. Then, the following identity holds on $\Sigma_t$:
\begin{equation}
 \label{KeyIdEqn}
\begin{split} 
&\emph{div}_{\Sigma_t}(2\Ker\frac{d\phi}{dt}\nabla_{\Sigma_t}\phi
+\abs{\nabla_{\Sigma_t}\phi}^2\nabla_{\Sigma_t}\Ker
-2\left\la\nabla_{\Sigma_t}\phi,\nabla_{\Sigma_t}\Ker\right\ra\nabla_{\Sigma_t}\phi\\
&+\mathcal{F}\Ker\phi\nabla_{\Sigma_t}\phi+\frac{1}{2}\phi^2\mathcal{F}
\nabla_{\Sigma_t}\Ker-\frac{1}{2}\phi^2\Ker\nabla_{\Sigma_t}\mathcal{F})
-\frac{d}{dt}\left(\abs{\nabla_{\Sigma_t}\phi}^2\Ker+\frac{1}{2}\phi^2\mathcal{F}\Ker\right)\\
=&-2\nabla^2_{\Sigma_t}\log\Ker\left(\nabla_{\Sigma_t}\phi,\nabla_{\Sigma_t}\phi\right)\Ker
-2\left(\frac{d\phi}{dt}-\left\la\nabla_{\Sigma_t}\log\Ker,\nabla_{\Sigma_t}\phi\right\rangle
+\frac{1}{2}\phi\mathcal{F}\right)^2\Ker\\
&+2\left(\frac{d\phi}{dt}-\left\la\nabla_{\Sigma_t}\log\Ker,\nabla_{\Sigma_t}\phi\right\ra
+\frac{1}{2}\phi\mathcal{F}\right)\left(\frac{d\phi}{dt}+\Delta_{\Sigma_t}\phi\right)\Ker\\
&-2H A\left(\nabla_{\Sigma_t}\phi,\nabla_{\Sigma_t}\phi\right)\Ker
-\frac{1}{2}\phi^2\left(\frac{d\mathcal{F}}{dt}+\Delta_{\Sigma_t}\mathcal{F}\right)\Ker.
\end{split}
\end{equation}
\end{lem}

\begin{proof}
First, note that, given functions $f$ and $g$ on a hypersurface $N$, 
\begin{equation}
\label{KeyIdEqn1}
f\Delta_Ng=g\Delta_Nf+\text{div}_N\left(f\nabla_Ng-g\nabla_Nf\right).
\end{equation}
Thus, a straightforward calculation gives that
\begin{equation}
\label{KeyIdEqn2}
\begin{split}
& 2\left(\frac{d\phi}{dt}-\la\nabla_{\Sigma_t}\log\Ker,\nabla_{\Sigma_t}\phi\ra
+\frac{1}{2}\phi\mathcal{F}\right)\left(\frac{d\phi}{dt}+\Delta_{\Sigma_t}\phi\right)\Ker\\
&-2\left(\frac{d\phi}{dt}-\la\nabla_{\Sigma_t}\log\Ker,\nabla_{\Sigma_t}\phi\ra
+\frac{1}{2}\phi\mathcal{F}\right)^2\Ker\\
=&\text{div}_{\Sigma_t}\left(2\Ker\frac{d\phi}{dt}\nabla_{\Sigma_t}\phi
+\phi\mathcal{F}\Ker\nabla_{\Sigma_t}\phi+\frac{1}{2}\phi^2\mathcal{F}\nabla_{\Sigma_t}\Ker
-\frac{1}{2}\phi^2\Ker\nabla_{\Sigma_t}\mathcal{F}\right)\\
&-\frac{d}{dt}\left(\abs{\nabla_{\Sigma_t}\phi}^2\Ker+\frac{1}{2}\phi^2\mathcal{F}\Ker\right)
+\frac{1}{2}\phi^2\Ker\left(\frac{d\mathcal{F}}{dt}+\Delta_{\Sigma_t}\mathcal{F}\right)
+\abs{\nabla_{\Sigma_t}\phi}^2\Delta_{\Sigma_t}\Ker\\
&-2\la\nabla_{\Sigma_t}\Ker,\nabla_{\Sigma_t}\phi\ra\Delta_{\Sigma_t}\phi
-2\la\nabla_{\Sigma_t}\log\Ker,\nabla_{\Sigma_t}\phi\ra^2\Ker\\
&+\left(\frac{d}{dt}\abs{\nabla_{\Sigma_t}\phi}^2-2\left\la\nabla_{\Sigma_t}\left(\frac{d\phi}{dt}\right)
,\nabla_{\Sigma_t}\phi\right\ra\right)\Ker.
\end{split}
\end{equation}

Note that,
\begin{equation}
\label{KeyIdEqn3}
\begin{split}
&2\left\la\nabla_{\Sigma_t}\Ker,\nabla_{\Sigma_t}\phi\right\ra\Delta_{\Sigma_t}\phi\\
=&\text{div}_{\Sigma_t}\left(2\la\nabla_{\Sigma_t}\Ker,\nabla_{\Sigma_t}\phi\ra
\nabla_{\Sigma_t}\phi\right)-2\nabla_{\Sigma_t}\phi
\left(\la\nabla_{\Sigma_t}\Ker,\nabla_{\Sigma_t}\phi\ra\right)\\
=&\text{div}_{\Sigma_t}\left(2\la\nabla_{\Sigma_t}\Ker,\nabla_{\Sigma_t}\phi\ra
\nabla_{\Sigma_t}\phi\right)-2\left\la\nabla_{\nabla_{\Sigma_t}\phi}
\nabla_{\Sigma_t}\Ker,\nabla_{\Sigma_t}\phi\right\ra\\
&-2\left\la\nabla_{\Sigma_t}\Ker,\nabla_{\nabla_{\Sigma_t}\phi}\nabla_{\Sigma_t}\phi\right\ra\\
=&\text{div}_{\Sigma_t}\left(2\la\nabla_{\Sigma_t}\Ker,\nabla_{\Sigma_t}\phi\ra
\nabla_{\Sigma_t}\phi\right)-2\nabla^2_{\Sigma_t}\Ker(\nabla_{\Sigma_t}\phi,\nabla_{\Sigma_t}\phi)\\
&-\left\la\nabla_{\Sigma_t}\Ker,\nabla_{\Sigma_t}\abs{\nabla_{\Sigma_t}\phi}^2\right\ra\\
\end{split}
\end{equation}
\begin{equation*}
\begin{split}
=&\text{div}_{\Sigma_t}\left(2\la\nabla_{\Sigma_t}\Ker,\nabla_{\Sigma_t}\phi\ra
\nabla_{\Sigma_t}\phi-\abs{\nabla_{\Sigma_t}\phi}^2\nabla_{\Sigma_t}\Ker\right)\\
&-2\nabla^2_{\Sigma_t}\Ker(\nabla_{\Sigma_t}\phi,\nabla_{\Sigma_t}\phi)
+\abs{\nabla_{\Sigma_t}\phi}^2\Delta_{\Sigma_t}\Ker.
\end{split}
\end{equation*}
Also, it is easy to check that
\begin{equation}
\label{KeyIdEqn4}
\nabla_{\Sigma_t}^2\log\Ker(\nabla_{\Sigma_t}\phi,\nabla_{\Sigma_t}\phi)\Ker
=\nabla_{\Sigma_t}^2\Ker(\nabla_{\Sigma_t}\phi,\nabla_{\Sigma_t}\phi)
-\la\nabla_{\Sigma_t}\log\Ker,\nabla_{\Sigma_t}\phi\ra^2\Ker.
\end{equation}

Finally, let $\Omega\subset\Real^n$ and $F:\Omega\times (0,1]\To\Real^{n+1}$ be the local parametrization of $\{\Sigma_t\}_{t\in (0,1]}$ such that $F(\Omega,t)\subset\Sigma_t$ and $dF/dt=H\textbf{n}$. Let $g_t$ be the pull back metric on $\Omega$ from $\Sigma_t$.
Thus, we have that
\begin{equation}
\label{KeyIdEqn5}
\begin{split}
&\frac{d}{dt}\abs{\nabla_{\Sigma_t}\phi}^2-2\left\la\nabla_{\Sigma_t}\left(\frac{d\phi}{dt}\right),
\nabla_{\Sigma_t}\phi\right\ra\\
=&\sum_{i,j}\partial_t g_t^{ij}\partial_i\phi\partial_j\phi
=2HA(\nabla_{\Sigma_t}\phi,\nabla_{\Sigma_t}\phi),
\end{split}
\end{equation}
where $\partial_t g_t^{ij}$ denotes the time derivative fixing a point in $\Omega$, and we use the following equation in the second equality (see the appendix B in \cite{Ec}):
\begin{equation}
\label{KeyIdEqn6}
\partial_t g_t^{ij}=2H\sum_{k,l}g_t^{ik}g_t^{jl}
A(\partial_k F,\partial_l F).
\end{equation}

Therefore, Lemma \ref{KeyIdLem} follows from (\ref{KeyIdEqn2})-(\ref{KeyIdEqn6}).
\end{proof}

Define $\Sigma_0$ to be the cone $C$. Note that, by Lemma \ref{PertLem} and the standard regularity theory for mean curvature flow, $\{\Sigma_t\setminus B_{2R_2}\}_{t\in [0,1]}$ is one parameter smooth family of hypersurfaces in $\Real^{n+1}$. We integrate the identity (\ref{KeyIdEqn}) over $\Sigma_t$ against the Hausdorff measure $d\mu_t$ on $\Sigma_t$ first, followed by the integration over $[0,T]$ against $dt$. Using the dominant convergence theorem, we conclude that

\begin{lem}
 \label{IntelKeyIdLem}
Given $R>2R_2$ and $0<T\le 1$, let $\Ker$ be a smooth positive function in 
\begin{equation}
Q_{R,T}=\left\{(x,t)\ \vline\ x\in\Sigma_t\setminus\bar{B}_R, t\in [0,T]\right\}.
\end{equation}
Assume that $\phi,\nabla_{\Sigma_t}\phi\in C^0_c(Q_{R,T})$ with $\phi(\cdot,0)\equiv 0$ and $\phi$ is smooth in the interior of $Q_{R,T}$. As in Lemma \ref{KeyIdLem}, set $\mathcal{F}=\left(d\Ker/dt-\Delta_{\Sigma_t}\Ker\right)/\Ker$. Then, we have the following identity:
\begin{equation}
 \label{IntelKeyIdEqn}
\begin{split}
 & \int_0^T\int_{\Sigma_t}\left(2\nabla_{\Sigma_t}^2\log\Ker+2HA+H^2g_t\right)
\left(\nabla_{\Sigma_t}\phi,\nabla_{\Sigma_t}\phi\right)\Ker d\mu_t dt\\
&+\int_0^T\int_{\Sigma_t}\frac{1}{2}\left(\frac{d\mathcal{F}}{dt}
+\Delta_{\Sigma_t}\mathcal{F}+H^2\mathcal{F}\right)\phi^2\Ker d\mu_t dt\\
\le & \int_0^T\int_{\Sigma_t}\left(\frac{d\phi}{dt}+\Delta_{\Sigma_t}\phi\right)^2\Ker d\mu_t dt
+\int_{\Sigma_T}\left(\abs{\phi}^2\mathcal{F}+\abs{\nabla_{\Sigma_T}\phi}^2\right)\Ker d\mu_T, 
\end{split}
\end{equation}
where $g_t$ is the metric on $\Sigma_t$ induced from $\Real^{n+1}$.
\end{lem}

Next, we choose the function $\Ker$ in the previous lemma, which is a suitable variation of the choice in \cite{ESS}. From now on, we fix a $\delta\in (0,1)$. Given $\alpha>0$, $0<T\le 1$ and $R>2R_2$, we set $\Ker$ in Lemma \ref{IntelKeyIdLem} to be
\begin{equation}
 \label{TestKerEqn}
\Ker_{\alpha,T,R}(x,t)=\exp\left[2\alpha(T-t)\left(\abs{x}^{1+\delta}
-R^{1+\delta}\right)+2\abs{x}^2\right].
\end{equation}
Similarly, set $\mathcal{F}_{\alpha,T,R}=\left(d\Ker_{\alpha,T,R}/dt-\Delta_{\Sigma_t}
\Ker_{\alpha,T,R}\right)/\Ker_{\alpha,T,R}$. Then, by straightforward calculations, we have that

\begin{lem}
 \label{PropTestKerLem}
There exist $\alpha_0>0$, $R_4>2R_2$ and $C_5>0$ such that for $\alpha\ge\alpha_0$, $0<T\le 1$, $R>R_4$ and $(x,t)\in\cup_{t\in [0,T]} \Sigma_t\setminus B_R\times\{t\}$,
\begin{align}
 \label{PropTestKerEqna}
&\mathcal{F}_{\alpha,T,R}(x,T)<0\\
\label{PropTestKerEqnb}
&\frac{d}{dt}\mathcal{F}_{\alpha,T,R}+\Delta_{\Sigma_t}\mathcal{F}_{\alpha,T,R}
+H^2\mathcal{F}_{\alpha,T,R}\ge 1\\
\label{PropTestKerEqnc}
& 2\nabla^2_{\Sigma_t}\log\Ker_{\alpha,T,R}+2HA+H^2g_t\ge C_5\emph{Id}.
\end{align}
\end{lem}

\begin{proof}
To prove the first property (\ref{PropTestKerEqna}), we compute $\mathcal{F}_{\alpha,T,R}$ explicitly as follows. First, note that,
\begin{equation}
\label{PropTestKerEqn1}
\mathcal{F}_{\alpha,T,R}=\frac{d}{dt}\log\Ker_{\alpha,T,R}
-\Delta_{\Sigma_t}\log\Ker_{\alpha,T,R}-\abs{\nabla_{\Sigma_t}\log\Ker_{\alpha,T,R}}^2.
\end{equation}
Then, we calculate each term on the right hand side of the equation (\ref{PropTestKerEqn1}) accordingly. Since $\Sigma$ is a self-shrinker of the mean curvature flow, we have that, for $x\in\Sigma_t$ and $0<t\le 1$,
\begin{equation}
\label{PropTestKerEqn2}
\partial_t x=H\textbf{n}
=\frac{\la x,\textbf{n}\ra}{2t}\ \textbf{n}.
\end{equation}
Thus,
\begin{equation}
\label{PropTestKerEqn3}
\begin{split}
&\frac{d}{dt}\log\Ker_{\alpha,T,R}\\
=&-2\alpha\left(\abs{x}^{1+\delta}-R^{1+\delta}\right)+2\alpha(1+\delta)(T-t)\abs{x}^{\delta-1}\left\la x,\partial_t x\right\ra+4\left\la x,\partial_t x\right\ra\\
=&-2\alpha\left(\abs{x}^{1+\delta}-R^{1+\delta}\right)+2\alpha(1+\delta)(T-t)\abs{x}^{\delta-1}H\la x,\textbf{n}\ra+4H\la x,\textbf{n}\ra\\
=&-2\alpha\left(\abs{x}^{1+\delta}-R^{1+\delta}\right)
+4\alpha(1+\delta)t(T-t)\abs{x}^{\delta-1}H^2
+8tH^2.
\end{split}
\end{equation}
Note that,
\begin{equation}
\label{PropTestKerEqn4}
\nabla_{\Sigma_t}\abs{x}^{\beta}
=\beta\abs{x}^{\beta-2}x^T,
\end{equation}
where $x^T$ is the tangential part of $x\in\Sigma_t$. Thus, by the equation (\ref{PropTestKerEqn2}),
\begin{equation}
\label{PropTestKerEqn5}
\begin{split}
\abs{\nabla_{\Sigma_t}\log\Ker_{\alpha,T,R}}^2
&=\left[2\alpha(1+\delta)(T-t)\abs{x}^{\delta-1}+4\right]^2\abs{x^T}^2\\
&=\left[2\alpha(1+\delta)(T-t)\abs{x}^{\delta-1}+4\right]^2\left(\abs{x}^2-4t^2H^2\right).
\end{split}
\end{equation}
It follows from (\ref{PropTestKerEqn2}) that
\begin{equation}
\label{PropTestKerEqn6}
\text{div}_{\Sigma_t}x^T=n-\la x,\textbf{n}\ra \text{div}_{\Sigma_t}\textbf{n}=n-2tH^2.
\end{equation}
Thus, by (\ref{PropTestKerEqn2}), (\ref{PropTestKerEqn4}) and (\ref{PropTestKerEqn6}), 
\begin{equation}
\label{PropTestKerEqn7}
\begin{split}
\Delta_{\Sigma_t}\abs{x}^\beta
&=\text{div}_{\Sigma_t}\left(\beta\abs{x}^{\beta-2}x^T\right)\\
&=\beta(\beta-2)\abs{x}^{\beta-4}\abs{x^T}^2
+\beta\abs{x}^{\beta-2}\text{div}_{\Sigma_t}x^T\\
&=\beta(\beta-2+n)\abs{x}^{\beta-2}-2\beta t\abs{x}^{\beta-2}H^2
-4\beta(\beta-2)t^2\abs{x}^{\beta-4}H^2.
\end{split}
\end{equation}
Hence, (\ref{PropTestKerEqn7}) with $\beta=1+\delta$ or $2$ gives that
\begin{equation}
\label{PropTestKerEqn8}
\begin{split}
&\Delta_{\Sigma_t}\log\Ker_{\alpha,T,R}\\
=&2\alpha(T-t)\Delta_{\Sigma_t}\abs{x}^{1+\delta}
+2\Delta_{\Sigma_t}\abs{x}^2\\
=&2\alpha(1+\delta)(\delta-1+n)(T-t)\abs{x}^{\delta-1}
-4\alpha(1+\delta)t(T-t)\abs{x}^{\delta-1}H^2\\
&+8\alpha(1-\delta^2)t^2(T-t)\abs{x}^{\delta-3}H^2
+4n-8tH^2.
\end{split}
\end{equation}
Therefore, combining the equations (\ref{PropTestKerEqn3}), (\ref{PropTestKerEqn5}) and (\ref{PropTestKerEqn8}), we conclude that
\begin{equation}
\label{PropTestKerEqn9}
\begin{split}
&\mathcal{F}_{\alpha,T,R}(x,t)\\
=&-2\alpha\left(\abs{x}^{1+\delta}-R^{1+\delta}\right)
-2\alpha(1+\delta)(\delta-1+n)(T-t)\abs{x}^{\delta-1}-4n\\
&-\left[2\alpha(1+\delta)(T-t)\abs{x}^{\delta}+4\abs{x}\right]^2
+16tH^2\\
&+8\alpha(1+\delta)t(T-t)\abs{x}^{\delta-1}H^2
-8\alpha(1-\delta^2)t^2(T-t)\abs{x}^{\delta-3}H^2\\
&+4t^2\left[2\alpha(1+\delta)(T-t)\abs{x}^{\delta-1}+4\right]^2H^2.
\end{split}
\end{equation}
In particular, by (\ref{PropTestKerEqn9}) and Lemma \ref{GeomLem}, there exists $\lam_1>0$, depending only on $C_1$ and $n$, such that: when $t=T$, $R>2R_2$ and $x\in\Sigma_T\setminus B_R$, 
\begin{equation}
\label{PropTestKerEqn10}
\mathcal{F}_{\alpha,T,R}(x,T)<-4n-16\abs{x}^2
+16tH^2+64t^2H^2\le -\lam_1R^{2}.
\end{equation}

Next, we compute $d\mathcal{F}_{\alpha,T,R}/dt$ and $\Delta_{\Sigma_t}\mathcal{F}_{\alpha,T,R}$ below. Namely,
\begin{equation}
\label{PropTestKerEqn11}
\begin{split}
\frac{d}{dt}\mathcal{F}_{\alpha,T,R}
=&-4\alpha(1+\delta)t\abs{x}^{\delta-1}H^2
+2\alpha(1+\delta)(\delta-1+n)\abs{x}^{\delta-1}\\
&+4\alpha(1+\delta)(\delta-1+n)(1-\delta)t(T-t)\abs{x}^{\delta-3}H^2\\
&+8\left[\alpha(1+\delta)(T-t)\abs{x}^\delta+2\abs{x}\right]
[\alpha(1+\delta)\abs{x}^\delta\\
&-2\alpha\delta(1+\delta)t(T-t)\abs{x}^{\delta-2}H^2
-4t\abs{x}^{-1}H^2]\\
&+16H^2+32tH\partial_t H
+8\alpha(1+\delta)(T-2t)\abs{x}^{\delta-1}H^2\\
&-16\alpha(1-\delta^2)t^2(T-t)\abs{x}^{\delta-3}H^4\\
&+16\alpha(1+\delta)t(T-t)\abs{x}^{\delta-1}H\partial_t H\\
&-8\alpha(1-\delta^2)(2Tt-3t^2)\abs{x}^{\delta-3}H^2\\
&+16\alpha(1-\delta^2)(3-\delta)t^3(T-t)\abs{x}^{\delta-5}H^4\\
&-16\alpha(1-\delta^2)t^2(T-t)\abs{x}^{\delta-3}H\partial_t H\\
&+8tH\left(H+t\partial_t H\right)
\left[2\alpha(1+\delta)(T-t)\abs{x}^{\delta-1}+4\right]^2\\
&-32t^2H^2\left[\alpha(1+\delta)(T-t)\abs{x}^{\delta-1}+2\right]
[\alpha(1+\delta)\abs{x}^{\delta-1}\\
&+2\alpha(1-\delta^2)t(T-t)\abs{x}^{\delta-3}H^2].
\end{split}
\end{equation}
Here $\partial_t H$ means the derivative with respect to a point moving perpendicularly to the hypersurface. Note that, by the computation in the appendix B of \cite{Ec},
\begin{equation}
\label{PropTestKerEqn12}
\partial_t H+\Delta_{\Sigma_t}H=-\abs{A}^2H.
\end{equation}
Thus, it follows from Lemma \ref{GeomLem} that, for $x\in\Sigma_t\setminus B_{2R_2}$ and $0<t\le 1$, 
\begin{equation}
\label{PropTestKerEqn13}
\abs{\partial_t H}\le\lam_2\abs{x}^{-3},
\end{equation}
where $\lam_2>0$ depends only on $C_1$. Hence, by Lemma \ref{GeomLem}, (\ref{PropTestKerEqn11}) and (\ref{PropTestKerEqn13}), there exist $\alpha_1>1$ and $r_1>2R_2$, depending only on $C_1$ and $R_2$, such that for $\alpha>\alpha_1$, $R>r_1$ and $x\in\Sigma_t\setminus B_R$,
\begin{equation}
\label{PropTestKerEqn14}
\frac{d}{dt}\mathcal{F}_{\alpha,T,R}(x,t)\ge 4\alpha(1+\delta)\abs{x}^{1+\delta}\left[\alpha(1+\delta)(T-t)\abs{x}^{\delta-1}+2\right].
\end{equation}
Similarly, we can estimate $\Delta_{\Sigma_t}\mathcal{F}_{\alpha,T,R}$. Note that,
\begin{equation}
\label{PropTestKerEqn15}
\Delta_{\Sigma_t}H^2=2H\Delta_{\Sigma_t}H+2\abs{\nabla_{\Sigma_t}H}^2,
\end{equation}
and for $\beta\neq 0$, 
\begin{equation}
\Delta_{\Sigma_t}(\abs{x}^\beta H^2)
=H^2\Delta_{\Sigma_t}\abs{x}^\beta
+2\left\la\nabla_{\Sigma_t}H^2,\nabla_{\Sigma_t}\abs{x}^\beta\right\ra
+\abs{x}^\beta\Delta_{\Sigma_t}H^2.
\end{equation}
Assuming that $R>2R_2$ and $x\in\Sigma_t\setminus B_R$, then Lemma \ref{GeomLem}, (\ref{PropTestKerEqn4}) and (\ref{PropTestKerEqn7}) give that, at $x$,
\begin{equation}
\label{PropTestKerEqn16}
\abs{\Delta_{\Sigma_t}H^2}\le\lam_3\abs{x}^{-4}\quad\text{and}\quad
\abs{\Delta_{\Sigma_t}\left(\abs{x}^\beta H^2\right)}\le\lam_4\abs{x}^{\beta-4},
\end{equation}
where $\lam_3>0$ depends only on $C_1$ and $\lam_4>0$ depends on $C_1$, $n$ and $\beta$. Since each term in $\mathcal{F}_{\alpha,T,R}$ is either $\abs{x}^\beta$, $H^2$, or $\abs{x}^\beta H^2$, a straightforward computation and (\ref{PropTestKerEqn16}) give that, for $x\in\Sigma_t\setminus B_{2R_2}$,
\begin{equation}
\label{PropTestKerEqn17}
\abs{\Delta_{\Sigma_t}\mathcal{F}_{\alpha,T,R}}\le \lam_5\alpha^2(T-t)\abs{x}^{2\delta-2}+\lam_5\alpha,
\end{equation}
where $\lam_5>0$ depends only on $C_1$, $R_2$, $\lam_3$ and $\lam_4$.
Also, for $x\in\Sigma_t\setminus B_{2R_2}$, it follows from Lemma \ref{GeomLem} that
\begin{equation}
\label{PropTestKerEqn18}
\abs{H^2\mathcal{F}_{\alpha,T,R}}\le\lam_6\alpha^2(T-t)\abs{x}^{2\delta-2}+\lam_6\alpha,
\end{equation}
where $\lam_6>0$ depends on $R_2$ and $C_1$. Hence, the second property (\ref{PropTestKerEqnb}) is verified by (\ref{PropTestKerEqn14}), (\ref{PropTestKerEqn17}) and (\ref{PropTestKerEqn18}), when $t\in (0,1]$.

Finally, we estimate $\nabla^2_{\Sigma_t}\abs{x}^\beta$ for $\beta\in (1,2]$. Fix $x\in\Sigma_t$, $0<t\le 1$. We choose a local geodesic orthonormal frame $\{e_1,\dots,e_n\}$ of $\Sigma_t$ at $x$. Thus, (\ref{PropTestKerEqn2}) gives that, at $x$,
\begin{equation}
\label{PropTestKerEqn19}
\begin{split}
&\nabla^2_{\Sigma_t}\abs{x}^\beta (e_i,e_j)
=\left\la\nabla_{e_i}\nabla_{\Sigma_t}\abs{x}^\beta,e_j\right\ra
=\beta\left\la\nabla_{e_i}\left(\abs{x}^{\beta-2}x^T\right),e_j\right\ra\\
=&\beta\left\la\nabla_{e_i}\left(\abs{x}^{\beta-2}x\right),e_j\right\ra
-\beta\abs{x}^{\beta-2}\left\la x,\textbf{n}\right\ra
\left\la\nabla_{e_i}\textbf{n},e_j\right\ra\\
=&\beta\left\la D\abs{x}^{\beta-2} ,e_i\right\ra\left\la x,e_j\right\ra
+\beta\abs{x}^{\beta-2}\left\la\nabla_{e_i}x,e_j\right\ra
+2\beta t\abs{x}^{\beta-2}HA(e_i,e_j)\\
=&\beta(\beta-2)\abs{x}^{\beta-4}\left\la x,e_i\right\ra\left\la x,e_j\right\ra
+\beta\abs{x}^{\beta-2}\delta_{ij}
+2\beta t\abs{x}^{\beta-2}HA(e_i,e_j).
\end{split}
\end{equation}

Now, given $\eta\in T_x\Sigma_t$, $\eta=\sum_k \eta_ke_k$ and at $x$, it follows from (\ref{PropTestKerEqn19}) that
\begin{equation}
\label{PropTestKerEqn20}
\begin{split}
&\nabla_{\Sigma_t}^2\abs{x}^{\beta}(\eta,\eta)\\
\ge & \beta\abs{x}^{\beta-4}\sum_{i,j}
\left[(\beta-2)\la x,e_i\ra\la x,e_j\ra
+\abs{x^T}^2\delta_{ij}\right]\eta_i\eta_j
+2\beta t\abs{x}^{\beta-2}HA(\eta,\eta)\\
=&\beta(\beta-2)\abs{x}^{\beta-4}\left(\sum_i \la x,e_i\ra\eta_i\right)^2+\beta\abs{x}^{\beta-4}\abs{x^T}^2\abs{\eta}^2
+2\beta t\abs{x}^{\beta-2}HA(\eta,\eta)\\
\ge & \beta(\beta-1)\abs{x}^{\beta-4}\abs{x^T}^2\abs{\eta}^2
+2\beta t\abs{x}^{\beta-2}HA(\eta,\eta)\\
=&\beta(\beta-1)\abs{x}^{\beta-2}\abs{\eta}^2
-4\beta(\beta-1)t^2\abs{x}^{\beta-4}H^2\abs{\eta}^2
+2\beta t\abs{x}^{\beta-2}HA(\eta,\eta).
\end{split}
\end{equation}
Hence, by Lemma \ref{GeomLem} and the assumption that $\beta\in (1,2]$, there exists $r_2>2R_2$, depending only on $\beta$, $C_1$ and $R_2$, such that for $x\in\Sigma_t\setminus B_{r_2}$, 
\begin{equation}
\label{PropTestKerEqn21}
\nabla_{\Sigma_t}^2\abs{x}^\beta(\eta,\eta)
\ge\frac{1}{2}\beta(\beta-1)\abs{x}^{\beta-2}\abs{\eta}^2.
\end{equation}

Hence, the third property (\ref{PropTestKerEqnc}) follows immediately from (\ref{PropTestKerEqn21}) with $\beta=1+\delta$ or $2$, and Lemma \ref{GeomLem}, when $t\in (0,1]$. Note that $\{\Sigma_t\setminus B_{2R_2}\}_{t\in [0,1]}$ is one parameter smooth (even at $t=0$) family of hypersurfaces, and $\Ker_{\alpha,T,R}$ and $\mathcal{F}_{\alpha,T,R}$ are smooth functions on $Q_{R,T}$. Therefore, the second and third properties in this lemma hold true for $t=0$.
\end{proof}

Finally, combining Lemma \ref{IntelKeyIdLem} and Lemma \ref{PropTestKerLem}, we establish the following Carleman inequality.

\begin{prop}
\label{CarlemanProp}
Let $\alpha>\alpha_0$, $0<T\le 1$, $R>R_4$ and $Q_{R,R}$ as in Lemma \ref{IntelKeyIdLem}. Assume that $\phi,\nabla_{\Sigma_t}\phi\in C^0_c(Q_{R,T})$ with $\phi(\cdot,0)\equiv 0$ and $\phi$ is smooth in the interior of $Q_{R,T}$. Then
\begin{equation}
 \label{CarlemanEqn}
\begin{split}
& \int_0^T\int_{\Sigma_t}\left(\abs{\phi}^2+C_5\abs{\nabla_{\Sigma_t}\phi}^2\right)\Ker_{\alpha,T,R} d\mu_t dt\\
\le & \int_0^T\int_{\Sigma_t}\left(\frac{d\phi}{dt}+\Delta_{\Sigma_t}\phi\right)^2\Ker_{\alpha,T,R}d\mu_t dt+\int_{\Sigma_T}\abs{\nabla_{\Sigma_T}\phi}^2\Ker_{\alpha,T,R}d\mu_T.
\end{split}
\end{equation}
\end{prop}

To conclude the proof of Theorem \ref{UniqueThm} by applying Proposition \ref{CarlemanProp}, we need to study the decay rates of $\abs{w}$ and its gradient $\abs{\nabla_{\Sigma_t}w}$.

\begin{lem}
\label{DecayLem}
There exist $R_5>R_3$ and $M>0$ such that for $x\in \Sigma_t\setminus B_{R_5}$ and $t\in [0,1]$,
\begin{equation}
 \label{DecayEqn}
\abs{w(x,t)}+\abs{\nabla_{\Sigma_t}w(x,t)}
\le\exp\left(-\frac{M\abs{x}^2}{t}\right).
\end{equation} 
\end{lem}

\begin{proof}
Fix $z_0\in C\setminus B_{2R_3}$. By Lemma \ref{PertLem}, for each $t\in(0,1]$, the component of $\Sigma_t\cap B_{\eps_0\abs{z_0}}(z_0)$ containing $z_0+U(z_0,t)\textbf{n}(z_0)$ can be written as the graph of a smooth function $u(\cdot,t)$ over the tangent plane $T_{z_0}C$ of $C$ at $z_0$ satisfying the property (\ref{PertEqn}). Moreover, there exists $\delta_1\in (0,\eps_0]$, depending only on $C_2$, such that the image of the orthogonal projection of $\Sigma_t\cap B_{\eps_0\abs{z_0}}(z_0)$ to $T_{z_0}C$, $0<t\le 1$ contains the disk in $T_{z_0}C$ centered at $z_0$ with radius $\delta_1\abs{z_0}$. We parametrize $T_{z_0}C$ by $F:\Real^n\To T_{z_0}C$, $F(p)=z_0+\sum_ip_ie_i$, where $p=(p_1,\dots,p_n)$ and $\{e_1,\dots,e_n\}$ is an orthonormal basis of $T_{z_0}C-z_0$. And we identify $u(p,t)$ with $u(F(p),t)$ and define $\bar{w}(p,t)=w(F(p)+u(p,t)\textbf{n}(z_0),t)$. Let $g_t$ be the pull back metric on $D_{\delta_1\abs{z_0}}$ from $\Sigma_t$ via the map $p\mapsto F(p)+u(p,t)\textbf{n}(z_0)$. In the following, $\partial_t$ denotes the partial derivative with respect to time $t$ fixing $p$. Then, it follows from Lemma \ref{JacobiLem} that, on $D_{\delta_1\abs{z_0}}\times (0,1]$,
\begin{equation}
\label{DecayEqn1}
\begin{split}
&\partial_t\bar{w}+\Delta_{g_t}\bar{w}\\
=&\frac{1}{2\sqrt{t}}\ v\left(\frac{x}{\sqrt{t}}\right)+\sqrt{t}\left\la \nabla_\Sigma v\left(\frac{x}{\sqrt{t}}\right),\frac{1}{\sqrt{t}}\frac{\partial x}{\partial t}-\frac{x}{2t\sqrt{t}}\right\ra
+\frac{1}{\sqrt{t}}\ \Delta_{\Sigma}v\left(\frac{x}{\sqrt{t}}\right)\\
=&\frac{1}{\sqrt{t}}\left[\frac{1}{2}v\left(\frac{x}{\sqrt{t}}\right)-\frac{1}{2}\left\la\nabla_\Sigma v\left(\frac{x}{\sqrt{t}}\right),\frac{x}{\sqrt{t}}\right\ra
+\Delta_{\Sigma}v\left(\frac{x}{\sqrt{t}}\right)\right]\\
&+\left\la\nabla_\Sigma v\left(\frac{x}{\sqrt{t}}\right),\frac{\partial x}{\partial t}\right\ra\\
\end{split}
\end{equation}
\begin{equation*}
=\frac{1}{\sqrt{t}}\left[-\abs{A}^2v\left(\frac{x}{\sqrt{t}}\right)
+Q\left(\frac{x}{\sqrt{t}},v,\nabla_\Sigma v,\nabla^2_\Sigma v\right)\right]
+\left\la\nabla_\Sigma v\left(\frac{x}{\sqrt{t}}\right),\frac{\partial x}{\partial t}\right\ra.
\end{equation*}
Here, we use the backward mean curvature flow equation for graphs, that is, for $x=F(p)+u(p,t)\textbf{n}(z_0)$ and $p\in D_{\delta_1\abs{z_0}}$,
\begin{equation}
\label{DecayEqn2}
\la\partial_t x,\textbf{n}\ra=H,
\end{equation}
and thus equivalently,
\begin{equation}
\label{DecayEqn3}
\partial_t u=-\sqrt{1+\abs{Du}^2}\ \text{div}\left(\frac{Du}{\sqrt{1+\abs{Du}^2}}\right).
\end{equation}
Hence, by the calculation (\ref{DecayEqn1}), Lemmas \ref{GeomLem}, \ref{PertLem} and \ref{JacobiLem}, we have that, on $D_{\delta_1\abs{z_0}}\times (0,1]$,
\begin{equation}
\label{DecayEqn4}
\abs{\partial_t\bar{w}+\Delta_{g_t}\bar{w}}\le\lam_1\abs{z_0}^{-2}\left(\abs{\bar{w}}
+\abs{z_0}\abs{\nabla_{g_t}\bar{w}}\right),
\end{equation} 
where $\lam_1>0$ depends only on $C_1$, $C_2$, $C_4$ and $\delta_1$. It is easy to verify that $g_t$ has the following properties: $g_0(0)=\text{Id}$,
\begin{equation}
\label{DecayEqn5.1}
\lam_2^{-1}\abs{\xi}^2\le \sum_{i,j}g_t^{ij}\xi_i\xi_j\le\lam_2\abs{\xi}^2
\quad\text{for all}\quad \xi=(\xi_1,\dots,\xi_n)\in\Real^n,
\end{equation}
\begin{equation}
\label{DecayEqn5.2}
\abs{Dg^{ij}_t}\le\lam_2\abs{z_0}^{-1}\quad\text{and}\quad
\abs{\partial_tg^{ij}_t}\le\lam_2\abs{z_0}^{-2}\quad \text{for}\quad 1\le i,j\le n,
\end{equation}
where $\lam_2>0$ depends only on $n$ and $C_2$. Thus, (\ref{DecayEqn4}) gives that
\begin{equation}
\label{DecayEqn6}
\abs{\partial_t\bar{w}+\text{div}\left(g^{-1}_tD\bar{w}\right)}\le\lam_3\abs{z_0}^{-2}
\left(\abs{\bar{w}}+\abs{z_0}\abs{D\bar{w}}\right),
\end{equation} 
where $\lam_3>0$ depends only on $\lam_1$ and $\lam_2$.

For the convenience of the readers, we present the arguments in page 2877-2879 of \cite{NgT} here to conclude the proof of Lemma \ref{DecayLem}. Let $R=\delta_1\abs{z_0}/2$ and $\bar{w}_R$ be a rescaling of $\bar{w}$, $\bar{w}_R(q,s)=\bar{w}(Rq,R^2s)$. Then $\bar{w}_R$ satisfies
\begin{equation}
\label{DecayEqn7}
\abs{\partial_s\bar{w}_R+\text{div}\left(g^{-1}_RD\bar{w}_R\right)}
\le\lam_3\left(\abs{\bar{w}_R}+\abs{D\bar{w}_R}\right),
\end{equation}
where $g^{-1}_R(q,s)=(g_R^{ij})=g^{-1}_{R^2s}(Rq)$. It follows from (\ref{DecayEqn5.2}) that, in $D_2\times [0,1/R^2]$,
\begin{equation}
\label{DecayEqn8}
\abs{Dg_R^{ij}}\le\lam_2\quad\text{and}\quad\abs{\partial_sg_R^{ij}}\le\lam_2
\quad\text{for}\quad 1\le i,j\le n.
\end{equation}
Furthermore, by Lemma \ref{GraphLem} and (\ref{DecayEqn5.1}), 
\begin{equation}
\label{DecayEqn9}
\abs{\bar{w}_R}+\abs{D\bar{w}_R}\le\lam_4\abs{z_0}^{-1}
\end{equation}
in $D_2\times [0,1/R^2]$, where $\lam_4>0$ depends only on $C_3$ and $\lam_2$. And $\bar{w}_R,D\bar{w}_R\in C^0(D_2\times [0,1/R^2])$ with $\bar{w}_R(\cdot,0)=0$, and $\bar{w}_R$ is smooth in $D_2\times (0,1/R^2]$. Define $\phi=\psi(q)\eta(s)\bar{w}_R$, where $\chi_{D_1}\le\psi\le\chi_{D_2}$ and $\chi_{[0,1/\alpha]}\le\eta\le\chi_{[0,2/\alpha)}$ are bump functions, and $\alpha\ge 2R^2$ is a positive constant to be chosen. Then
\begin{equation}
\label{DecayEqn10}
\abs{\partial_s\phi+\text{div}\left(g^{-1}_RD\phi\right)}
\le\lam_3\left(\abs{\phi}+\abs{D\phi}\right)+\lam_5\alpha\left(\abs{\bar{w}_R}+\abs{D\bar{w}_R}\right)\chi_E.
\end{equation}
Here $E=\left(D_2\times [0,2/\alpha)\right)\setminus\left(D_1\times [0,1/\alpha]\right)$, and $\lam_5>0$ depends only on $R_3$, $\delta_1$, $\lam_2$ and $\lam_3$. Hence, using the Carleman inequality in \cite{EF} (see also Lemma 2.1 in \cite{NgT}), there exists $M_1>1$, depending only on $n$ and $\lam_2$, such that for $0<a<1/\alpha$,
\begin{equation}
\label{DecayEqn11}
\begin{split}
&\int_{\Real^{n+1}}\left(\alpha^2\phi^2+\alpha\sigma_a\abs{D\phi}^2\right)\sigma_a^{-\alpha}\Ker_adqds\\
\le& \alpha^\alpha M_1^\alpha\sup_{s\ge 0}\int_{\Real^n\times\{s\}}\left(\phi^2+\abs{D\phi}^2\right)dq
+M_1\lam_3^2\int_{\Real^{n+1}}\left(\abs{\phi}+\abs{D\phi}\right)^2\sigma_a^{1-\alpha}\Ker_adqds\\
&+M_1\lam_5^2\alpha^2\int_E\left(\abs{\bar{w}_R}+\abs{D\bar{w}_R}\right)^2\sigma_a^{1-\alpha}\Ker_adqds.
\end{split}
\end{equation}
Here $\Ker_a(q,s)=(s+a)^{-n/2}\exp\left[-\abs{q}^2/4(s+a)\right]$, $\sigma_a(s)=\sigma(s+a)$ and $\sigma:(0,4/\alpha)\To (0,+\infty)$ satisfying that $M_1^{-1}s\le\sigma(s)\le s$. If $\alpha\ge 2M_1\lam_3$, then the second term on the right hand side can be absorbed by the left hand. In $E$, $\sigma_a^{-\alpha}\Ker_a\le\alpha^{\alpha+\frac{n}{2}}M_1^{\alpha}$. Hence, by (\ref{DecayEqn9}), we get
\begin{equation}
\label{DecayEqn12}
\int_{\Real^{n+1}}\left(\alpha^2\phi^2+\alpha\sigma_a\abs{D\phi}^2\right)\sigma_a^{-\alpha}\Ker_adqds
\le\alpha^{\alpha+\frac{n}{2}}M_2^\alpha,
\end{equation}
where $M_2>M_1$ depends only on $R_3$, $\lam_4$, $\lam_5$ and $M_1$. Let $\rho=1/(M_2e)$ and $a=\rho^2/(2\alpha)$. Then, in $D_{2\rho}\times [0,\rho^2/(2\alpha)]$, 
\begin{equation}
\label{DecayEqn13}
\sigma_a^{1-\alpha}\Ker_a\ge \alpha^{\alpha+\frac{n}{2}-1}M_2^{2\alpha+n-2}.
\end{equation}
Therefore, we deduce from (\ref{DecayEqn12}) that
\begin{equation}
\label{DecayEqn14}
\int_{D_{2\rho}\times [0,\rho^2/(2\alpha)]}\left(\phi^2+\abs{D\phi}^2\right)dqds\le M_2^{2-\alpha-n}.
\end{equation}
We now choose $\alpha=M_3R^2$, where $M_3>1$ depends only on $R_3$, $\delta_1$, $\lam_3$ and $M_2$, so that
\begin{equation}
\label{DecayEqn15}
\int_{D_{2\rho}\times [0,\rho^2/(2M_3R^2)]}\left(\phi^2+\abs{D\phi}^2\right)dqds\le M_2^{2-n}\text{e}^{-2R^2}.
\end{equation}
By Lemma 4.1 in \cite{NgT}, this implies that, in $D_\rho\times [0,\rho^2/(4M_3R^2)]$,
\begin{equation}
\abs{\phi}+\abs{D\phi}\le\lam_6R^c\text{e}^{-R^2},
\end{equation}
where $\lam_6>0$ depends on $n$, $\lam_2$ and $\lam_3$, and $c>0$ depends on $n$. Undoing the change of variables, we get
\begin{equation}
\label{DecayEqn16}
\abs{w(z_0,t)}+\abs{\nabla_{\Sigma_t}w(z_0,t)}\le\lam_6\abs{z_0}^c\text{e}^{-\delta_1^2\abs{z_0}^2},
\end{equation}
if $0\le t\le \rho^2/(4M_3)$.

Therefore, It follows from the definition of $w$ that, there exist $M_4>0$ and $r_1>0$, depending only on $R_3$, $\delta_1$, $M_2$, $M_3$, $\lam_6$ and $c$, such that for $x\in\Sigma\setminus B_{r_1}$, 
\begin{equation}
\label{DecayEqn17}
\abs{v(x)}+\abs{\nabla_\Sigma v(x)}\le\exp\left(-M_4\abs{x}^2\right).
\end{equation}
It is clear that Lemma \ref{DecayLem} follows immediately from the inequality (\ref{DecayEqn17}).
\end{proof}

Now, we are ready to conclude the proof of Theorem \ref{UniqueThm}.
 
\begin{proof}
(of Theorem \ref{UniqueThm}) 
We basically follow the argument in \cite{ESS}. Choose $\eps_1$ in Lemma \ref{HeatLem} to be $(1+C_5)/4$. Then, choose $R$ large such that each lemma and proposition can be applied, and $T=M/16$. 

Given $a\in (0,1)$ and $r\gg 1$, we consider $\psi_{a,r}\in C_c^\infty(\Real^{n+1})$ satisfying that  $\psi_{a,r}\equiv 1$ when $(1+2a)R\le\abs{x}\le r$, $\psi_{a,r}\equiv 0$ when $\abs{x}\le (1+a)R$ or $\abs{x}\ge 2r$, $0\le\psi_{a,r}\le 1$ and $\abs{D\psi_{a,r}}$, $\abs{D^2\psi_{a,r}}$ are bounded from above by a function of $a$ and $R$. Choose the test function $\phi$ in Proposition \ref{CarlemanProp} to be $\phi_{a,r}=\psi_{a,r}w$. Thus, we get that
\begin{equation}
 \label{TestHeatEqn1}
\begin{split}
&\frac{d\phi_{a,r}}{dt}+\Delta_{\Sigma_t}\phi_{a,r}\\
=&\psi_{a,r}\left(\frac{dw}{dt}
+\Delta_{\Sigma_t}w\right)+2\left\la\nabla_{\Sigma_t}\psi_{a,r},\nabla_{\Sigma_t}w\right\rangle
+\left(\left\la D\psi_{a,r},\partial_t x\right\ra
+\Delta_{\Sigma_t}\psi_{a,r}\right)w.
\end{split}
\end{equation}
Hence, by Lemma \ref{HeatLem}, 
\begin{equation}
 \label{TestHeatEqn2}
\begin{split}
& \abs{\frac{d\phi_{a,r}}{dt}+\Delta_{\Sigma_t}\phi_{a,r}}^2\\
\le & \frac{1}{2}\left(\abs{\phi_{a,r}}^2+C_5\abs{\nabla_{\Sigma_t}\phi_{a,r}}^2\right)
+\lam\left(\abs{w}^2+\abs{\nabla_{\Sigma_t}w}^2\right)\chi_{a,R,r},
\end{split}
\end{equation}
where $\lam>0$ depends only on $n$, $C_1$, $a$ and $R$, and $\chi_{a,R,r}$ is the characteristic function supported in $\left(B_{2r}\setminus B_r\right)\cup \left(B_{(1+2a)R}\setminus B_{(1+a)R}\right)$.

Hence, by Proposition \ref{CarlemanProp} and the inequality (\ref{TestHeatEqn2}), for $\alpha>\alpha_0$,
\begin{equation}
 \label{CarlemanEqn1}
\begin{split}
& \int_0^T\int_{\Sigma_t}\phi_{a,r}^2\Ker_{\alpha,T,R}d\mu_t dt\\
\lesssim & \int_0^T\int_{\Sigma_t\cap B_{(1+2a)R}\setminus B_{(1+a)R}}\left(\abs{w}^2+\abs{\nabla_{\Sigma_t}w}^2\right)\Ker_{\alpha,T,R}d\mu_t dt\\
&+\int_0^T\int_{\Sigma_t\cap B_{2r}\setminus B_r}\left(\abs{w}^2+\abs{\nabla_{\Sigma_t}w}^2\right)\Ker_{\alpha,T,R}d\mu_t dt\\
&+\int_{\Sigma_T}\abs{\nabla_{\Sigma_T}\phi_{a,r}}^2\Ker_{\alpha,T,R}d\mu_T,
\end{split}
\end{equation}
where $\lesssim$ stands for being less than some positive multiple of the followed quantities, which depends only on $n$, $\lam$, $a$, $R$ and $T$. Furthermore, it follows from the fact that $M/T>8$, Lemma \ref{DecayLem} and the inequality (\ref{CarlemanEqn1}) that
\begin{equation}
 \label{CarlemanEqn2}
\begin{split}
 & \exp\left\{\left[(1+17a)^{1+\delta}-1\right]\alpha TR^{1+\delta}\right\}\int_0^{T/2}\int_{\Sigma_t\cap B_r\setminus B_{(1+17a)R}}w^2d\mu_t dt\\
\lesssim & r^2\exp\left(2^{2+\delta}\alpha Tr^{1+\delta}-\frac{Mr^2}{2T}\right)+\exp\left\{2\left[(1+2a)^{1+\delta}-1\right]\alpha TR^{1+\delta}\right\}+1.
\end{split}
\end{equation}

Note that $(1+17a)^{1+\delta}-1\ge 17a$ and $(1+2a)^{1+\delta}-1\le 8a$. Thus, letting $r\To\infty$, 
\begin{equation}
 \label{CarlemanEqn3}
\int_0^{T/2}\int_{\Sigma_t\setminus B_{(1+17a)R}}w^2d\mu_t dt\lesssim \exp\left(-\alpha aTR^{1+\delta}\right).
\end{equation}
Then, let $\alpha\To\infty$ and thus,
\begin{equation}
 \label{ConcEqn}
\int_0^{T/2}\int_{\Sigma_t\setminus B_{(1+17a)R}}w^2d\mu_t dt\le 0.
\end{equation}
By the arbitrariness of $a$, we conclude that $w\equiv 0$ on $Q_{R,T/2}$ and thus $v\equiv 0$ on $\Sigma\setminus B_{2R/\sqrt{T}}$.

Therefore, Theorem \ref{UniqueThm} follows immediately  by applying the strong continuation theorem (cf. \cite{GL1,GL2}) to $L_\Sigma$ inside the compact set, and an open and closed argument.
\end{proof}

\bibliographystyle{Plain}

\end{document}